\documentclass[12pt]{amsart}
\usepackage{amscd}
 \pdfoutput=1
 
\usepackage{amsmath}
\usepackage{amsfonts}
\usepackage{amssymb}
\usepackage{amsthm}
\usepackage{mathtools}

\usepackage[utf8]{inputenc}

	\usepackage{pinlabel}
	\usepackage{lmodern}
	\usepackage{textcomp}
	\usepackage[hidelinks]{hyperref}
	\usepackage{color}		
	\usepackage{tabu}
	
	\usepackage{graphicx} 
    \graphicspath{{bilder/}}
    \usepackage{calc}
    
    \usepackage{tikz-cd} 

\newtheorem{theorem}{Theorem}[section]
\newtheorem{lemma}[theorem]{Lemma}
\newtheorem{proposition}[theorem]{Proposition}
\newtheorem*{proposition*}{Proposition}
\newtheorem*{theorem*}{Theorem}
\newtheorem{corollary}[theorem]{Corollary}
\newtheorem{definition}[theorem]{Definition}

\theoremstyle{definition}
\newtheorem{remark}[theorem]{Remark}

\numberwithin{equation}{section}

\newcommand{\Z}{\mathbb{Z}}
\newcommand{\R}{\mathbb{R}}
\newcommand{\C}{\mathbb{C}}
\newcommand{\Q}{\mathbb{Q}}
\newcommand{\Cp}{\mathbb{CP}}

\newcommand{\Ker}{\operatorname{Ker}}
\newcommand{\dbar}{\bar{\partial}}
\newcommand{\M}{\mathcal{M}}

\newcommand{\st}{\text{st}}
\newcommand{\Ha}{\mathcal{H}}

\newcommand{\sub}{\text{sub}}
\newcommand{\Al}{\mathcal{A}}
\newcommand{\ind}{\operatorname{ind}}
\newcommand{\con}{\operatorname{con}}

\makeatletter
\newcommand{\labitem}[2]{%
\def\@itemlabel{{#1}}
\item
\def\@currentlabel{#1}\label{#2}}
\makeatother

\begin{document}

\title[Lch for attaching links in higher dimensions]
{Legendrian contact homology for attaching links in higher dimensional subcritical Weinstein manifolds}
\author{Cecilia Karlsson}

\address{University of Oslo,
 Department of Mathematics,\\
Postboks 1053, Blindern, 
0316 Oslo, Norway\\
cecikarl@math.uio.no}

\begin{abstract}
Let $\Lambda$ be a  link of Legendrian spheres in the boundary of a subcritical $2n$-dimensional Weinstein manifold $X$. We show that, under some geometrical assumptions, the computation of  the Legendrian contact homology of $\Lambda$ can be reduced to a computation of Legendrian contact homology in 1--jet spaces. Since the Legendrian contact homology in 1--jet spaces is well studied, this gives a simplified way to compute the Legendrian contact homology of $\Lambda$. 
%

We restrict to the case when the attaching spheres of the subcritical handles of $X$ do not interact with each other, and we assume that there are no handles of index $n-1$. Moreover, we will only consider mod 2 coefficients for now. The more general situation will be addressed in a forthcoming paper.

As an application we compute the homology of the free loop space of $\Cp^2$.
\end{abstract}

\maketitle
\section{Introduction}
A \emph{Weinstein manifold} is the symplectic counterpart of a Stein manifold in complex geometry. More precisely, any Weinstein manifold $X$ of dimension $2n$ can be given a handle decomposition into symplectic handles of index at most $n$. The handles are attached along isotropic spheres in the contact boundary of $X$, and after having attached the handles of index less than $n$ we get a \emph{subcritical Weinstein manifold}. The top index handles are then attached along Legendrian spheres in the contact boundary of the subcritical part of $X$. It has been shown that much of the symplectic topology of $X$ is encoded in the Legendrian attaching spheres. Indeed, the symplectic homology of any subcritical Weinstein manifold vanishes, and by \cite{surg1} it follows that the symplectic homology of $X$ is isomorphic to the Hochschild homology of the Chekanov-Eliashberg DGA of the Legendrian attaching link. 

The \emph{Chekanov-Eliashberg DGA} $\Al(V, \Lambda)$ of a Legendrian $\Lambda$ in a contact manifold $(V,\lambda)$ is freely generated by the \emph{Reeb chords} of $\Lambda$. These are solution curves to the Reeb vector field $R_\lambda$ associated to $\lambda$, defined by $\lambda(R_\lambda) =1, d\lambda(R_\lambda, \cdot)=0$, and the Reeb chords should have start and end point on $\Lambda$. The grading is given by a Maslov-type index, and the differential counts certain pseudo-holomorphic curves. In the case when $(V, \lambda)$ is the contact boundary of a Weinstein manifold $X$, this is given by a count of pseudo-holomorphic disks in $\R \times V$, capped off with pseudo-holomorphic planes in $X$. Such disks are called \emph{anchored in $X$}. See Section \ref{sec:fillable}. Another special case is when $V$ is the 1--jet space $J^1(M)$ of a smooth manifold $M$. Then one counts pseudo-holomorphic disks either in the symplectization of $J^1(M)$ or in the Lagrangian projection $T^*M$. This case is rather well-studied, and there are a number of computational tools available, even in higher dimensions. See e.g.\ \cite{legsub,PR,trees,orienttrees,korta}. 

In this paper we describe a setup where the Chekanov-Eliashberg DGA of the attaching spheres in the boundary of a subcritical Weinstein manifold $X$ with $c_1(X)=0$ can be computed from Legendrians in 1--jet spaces. In particular, we do not have to consider pseudo-holomorphic disks anchored in $X$.

This is a generalization of the work in \cite{en}, where the Chekanov-Eliashberg DGA is computed in the boundary of subcritical Weinstein 4-manifolds. We will assume that $n>2$, and we focus on a simplified situation where the attaching spheres of the subcritical handles do not interact with each other, and where we do not have any handles of index $n-1$. 
 We also restrict to $\Z_2$-coefficients. 
 The more general situation will be dealt with in a forthcoming paper, together with a careful treatment of signs so that we can compute the Legendrian contact homology of the attaching link over $\Z$.

To obtain our result, we need $\Lambda$ to satisfy some assumptions when passing through the  subcritical handles. Namely, if $\Lambda$ passes through a handle of index $k$ we assume it to be of the form $D^k \times \Lambda_\sub$ in the handle. Here $D^k$ is the core of the subcritical handle and $\Lambda_\sub \subset S^{2n-2k-1}$ is a Legendrian submanifold with respect to the standard contact structure. We also assume that $\Lambda$ is contained in a 1--jet neighborhood of $\Lambda_\st = D^k \times \Lambda_{\st,\sub}$ when passing through the handle, where $\Lambda_{\st,\sub} \subset S^{2n-2k-1} = \{z 
\in \C^{2n-2k}; |z| = 1\}$ is the standard Legendrian unknot, given by the real part of $S^{2n-2k-1}$.
%
%
In addition we assume that the part of $\Lambda$ outside the sub-critical handles is contained in a Darboux ball  $D_a \subset S^{2n-1}$, which we then identify with a ball in $J^1(\R^{n-1})$. 

In this way we cover $\Lambda$ with charts of Legendrians in 1--jet spaces. This is in general not enough to be able to compute the Chekanov-Eliashberg DGA of $\Lambda$ in 1--jet spaces, since we might have pseudo-holomorphic disks that leave the Darboux ball and the 1--jet neighborhood of $\Lambda_\st$. To remedy this problem, we Legendrian isotope $\Lambda$ in a neighborhood of the attaching regions for the subcritical handles, by performing a high-dimensional analogue of the dipping procedure in \cite{sabloffdipp}.

As a result we get two sub-DGAs $\Al_H, \Al_D$ of the Chekanov-Eliashberg DGA of the isotoped $\Lambda$, which are described in detail in Section \ref{sec:differential}. Briefly, $\Al_H$ gives the DGA in the subcritical handles, and is a free product of Chekanov-Eliashberg DGAs, one for each subcritical handle, and where each of those can be computed in the corresponding $J^1(\Lambda_\st)$ from a dipped version of 
 $D^k \times \Lambda_\sub$. The sub-DGA $\Al_D$ is given by $\Al(J^1(\R^{n-1}), \Lambda \cap D_a)$. Moreover, $\Al_D$ and $\Al_H$ contain a common sub-DGA $\Al_S$, described in Section \ref{sec:differential}.  


The main result of the paper is the following.
\begin{theorem}\label{thm:main}
Assume that $V$ and $\Lambda$ satisfy the geometric assumptions stated above. Then the DGA  $\Al(V,\Lambda)$ is quasi-isomorphic to the pushout of the diagram
 \begin{equation*}
\begin{tikzcd} 
	\Al_S \arrow {d}{i} \arrow {r}{i} & \Al_H  \\
	\Al_D, & 
\end{tikzcd}
\end{equation*} 
 where $i$ is the inclusion and where $\Al_S, \Al_D, \Al_H$ can be computed from Legendrians in one-jet spaces.
\end{theorem}

As an application we describe a Weinstein handle decomposition of $T^*\Cp^2$ and compute the Chekanov-Eliashberg DGA of the index $4$ attaching sphere.  
Using the relation between the Legendrian contact homology of the attaching spheres and the symplectic homology of the resulting Weinstein manifold \cite{surg1} together with the results of \cite{cotangentloop,cotangentloopviterbo, cotangentloopweber}, which relate the symplectic homology of $T^*M$ with the singular homology of the free loop space of $M$, this gives a description of the singular homology of the free loop space of $\Cp^2$.

From \cite{surg1} it also follows that there is a relation between the Chekanov-Eliashberg DGA of the Legendrian attaching spheres and the wrapped Fukaya category of the cocores of the critical handles. By recent results in \cite{gapash1, geomgen} these cocores generate the wrapped Fukaya category of the resulting Weinstein manifold. In \cite{rm} the authors use this together with the formula in \cite{en} to give examples of mirror manifolds in homological mirror symmetry. Similar calculations are performed in \cite{lauraemmy}. We hope that such computations can be made in higher dimensions with the help of our work. We also hope that one can use our results to perform higher-dimensional analogues of the computations in \cite{etgulekili1, etgulekili2}, where the authors use Kozul duality together with the work in \cite{en} to compute the wrapped Fukaya category for 4-dimensional plumbings.

 \subsection*{Outline}

 In Section \ref{sec:background} we fix notation, give a brief introduction to Weinstein manifolds and define Legendrian contact homology in contact manifolds which are Weinstein fillable. 
 We also describe the easier case when the contact manifold is the 1--jet space of a smooth manifold. 
 In Section \ref{sec:dga} we explain the assumptions needed for us to show that Legendrian contact homology in the boundary of a subcritical Weinstein manifold reduces to a computation of Legendrian contact homology in some different 1--jet spaces. We also describe the dipping procedure and give a more careful statement of Theorem \ref{thm:main}. 
 In Section \ref{sec:setup} we give proofs of the results in Section \ref{sec:dga}. 
 In Section \ref{sec:examples} we use a Weinstein handle decomposition of $T^*\Cp^2$ to compute the singular homology of the free loop space of $\Cp^2$.

\subsection*{Acknowledgments}
We thank Tobias Ekholm and John Rognes for insightful discussions.

\section{Background}\label{sec:background}
A \emph{Weinstein manifold} is a symplectic manifold $(X, \omega)$ equipped with a Liouville vector field $Z$ and a Morse function for which $Z$ is gradient-like. Along the boundary $V$ of $X$ we get an induced contact structure with contact form $\lambda = \iota_{Z}\omega$. The Morse function allows us to give a handle decomposition of $X$ into Weinstein handles, defined below, and where the handles are attached along isotropic spheres in the contact boundary. If $\dim X = 2n$, these handles are of index at most $n$, and if $X$ only have handles of index less than $n$ we say that $X$ is \emph{subcritical}. A contact manifold $(V, \lambda)$ that occurs as the boundary of some Weinstein manifold as above is called \emph{Weinstein fillable}.

\subsection{Notation}
For  $u = (u_1, \dotsc, u_l) \in \R^l$ we write 
\begin{align*}
 du &= (du_1, \dotsc, du_l) \in (T^*\R^l)^l \\
 \partial_u &= (\partial_{u_1}, \dotsc, \partial_{u_l}) \in (T\R^l)^l
\end{align*}
and if $du, dv \in (T^*\R^l)^l$ we write 
\begin{align*}
 vdu &=  \sum_{i=1}^l v_idu_i \\
 du \wedge dv &=  \sum_{i=1}^l du_i \wedge dv_i. 
\end{align*}

If $M$ is a $k$-dimensional smooth manifold we write $(u,v,r) = (u_1, \dotsc , u_k,v_1, \dotsc, v_k,r)$ for the coordinates of the 1--jet space $J^1(M) = T^*M \times \R$ of $M$, where $u$ are the coordinates on $M$, $v$ are the cotangent coordinates and $r$ is the coordinate in the $\R$-direction. 

\subsection{Geometry}\label{sec:geometry}
Let $X$ be a subcritical Weinstein manifold of dimension $2n$, let $V$ be its contact boundary. Assume that $X$ admits a Weinstein handle decomposition 
\begin{equation*}
X = B^{2n} \cup \bigcup_{k=1}^{n-1} \bigcup_{i=1}^{l_k}\Ha_i^k,
\end{equation*}
where $\Ha_1^k, \dotsc, \Ha_{l_k}^k$ are Weinstein handles of index $k$. We will use the following model of such a handle:
\begin{equation*}
 \Ha^k =\{(x,y,p,q) \in \R^{2n-2k} \times \R^{2k} : -\delta^2 \leq \sum_{j=1}^{n-k}\frac{1}{2}(x_j^2 + y_j^2) +  \sum_{j=1}^k 2p_j^2 - q_j^2\leq \delta^2\},
\end{equation*}
where $\delta > 0 $ is some small constant. 
 
The handle has a Liouville vector field
\begin{equation}
 Z = \sum_{j=1}^{n-k}\frac{1}{2}(x_j \partial_{x_j} + y_j\partial_{y_j}) + \sum_{j=1}^k  2p_j\partial_{p_j} - q_j\partial_{q_j}
\end{equation} 
with respect to the standard symplectic form $\omega_\st = dx \wedge dy + dp \wedge dq$,
and this vector field is transverse to the boundary
\begin{equation*}
\Ha_{\pm}^k= \{ (x,y,p,q) \in \Ha^k; \sum_{j=1}^{n-k}\frac{1}{2}(x_j^2 + y_j^2) +  \sum_{j=1}^k 2p_j^2 - q_j^2 = \pm \delta^2\} ,
\end{equation*}
and points out of $\Ha^k$ along $\Ha_+^k$ and into $\Ha^k$ along $\Ha_{-}^k$. 
If now $X'$ is a Weinstein manifold with contact boundary $V'$, and if $\Upsilon' \subset V'$ is an isotropic sphere of dimension $k-1$, then one can use the Liouville vector field of $X'$ along a neighborhood of $\Upsilon'$ in $V'$ and the Liouville vector field of $\Ha^k$ along $\Ha_-^k$ to attach $\Ha^k$ to $X'$ to get a new Weinstein manifold $X = X' \cup \Ha^k$. See \cite{weinstein}. The boundary of $X$ will again be a contact manifold, where the contact form on $\Ha_{+}^k$ is given by
\begin{equation}
 \alpha_{+} = \omega_{st}(Z,\bullet)|_{\Ha_{+}^k} = \sum_{j=1}^{n-k}\frac{1}{2}(x_j d{y_j} - y_jd{x_j}) +  \sum_{j=1}^k 2p_jd{q_j} + q_jd{p_j}.
\end{equation}

\begin{remark}
 If $(V,\lambda)$ is the contact boundary of a subcritical Weinstein manifold then we will identify it with standard contact $S^{2n-1}$ with handles attached. That is 
\begin{equation*}
V = S^{2n-1} \cup \bigcup_{k=1}^{n-1} \bigcup_{i=1}^{l_k}\Ha_{i,+}^k.
\end{equation*}  
\end{remark}

\subsection{LCH in fillable contact manifolds}\label{sec:fillable}
Here we give a brief overview of the definition of Legendrian contact homology in the boundary of a Weinstein manifold. We refer to  \cite{sft, surg1, tobdisk} for a more thorough treatment. 
Let $(V, \lambda)$ be the contact boundary of a Weinstein manifold $X$, and let $\Lambda \subset V$ be a Legendrian submanifold. Assume that $c_1(X) = 0$. Then we define the Legendrian contact homology (LCH) of $\Lambda$ to be the homology of the differential graded algebra (DGA) $\Al(V, \Lambda)$ which is defined as follows.

The algebra is freely generated over $\Z_2[H_2(X,\Lambda)]$ by the \emph{Reeb chords of $\Lambda$}, which are solution curves of the Reeb vector field $R_\lambda$ having start and end point on $\Lambda$. The chords are graded by a Maslov type index, called the \emph{Conley-Zehnder index} $\mu_{cz}$. That is, if $c$ is a Reeb chord of $\Lambda$ then the grading of $c$ is given by
\begin{equation}\label{eq:czgrading}
 |c| = \mu_{cz}(\gamma_c) -1,
\end{equation} 
where $\gamma_c$ is a closed path in $V$ from the end point $c_+$ of $c$ in $\Lambda$, going through $\Lambda$ to the starting point $c_-$ of $c$ and then follows $c$ to the end point $c_+$ in the case when $c_-$ and $c_+$ belong to the same component of $\Lambda$. In the case when the start and end point belong to different components of $\Lambda$ the path $\gamma_c$ also contains a path from the component of $\Lambda$ containing $c_+$ to the component of $\Lambda$ containing $c_-$. The Conley-Zehnder index measures how much the contact distribution $\xi = \Ker \lambda$ rotates along this path.

The differential is defined by a count of \emph{anchored pseudo-holomorphic curves} in $X$, as follows. Let $J$ be an almost complex structure on $X$ which is compatible with the symplectic form and which is cylindrical in a neighborhood $\R_t \times V$ of the boundary $V$, meaning that it is invariant under translations in the $\R$--factor, gives a complex structure on $\Ker \lambda$ and satisfies that $J(\partial_t)=R_\lambda$.

An  \emph{anchored pseudo-holomorphic disk} is a two-level $J$-holomorphic building, where the top level is given by a $J$-holomorphic map
\begin{equation*} 
 u \colon (D_m, \partial D_m) \to (\R \times V, \R \times \Lambda),  
\end{equation*}
where $D_m$ is the unit disk in $\C$ with $m$ punctures $p_0, p_1, \dotsc, p_{m-1}$ along its boundary, where $p_0$ is distinguished and located at $1$. This puncture is called \emph{positive} and the punctures $p_1, \dotsc, p_{m-1}$ are called \emph{negative}. Near the puncture $p_0$ the map $u$ is required to be asymptotic to a Reeb chord $a$ of $\Lambda$ at $+\infty$, and near the negative puncture $p_i$ it should be asymptotic to a Reeb chord $b_i$ of $\Lambda$ at $-\infty$ for $i = 1, \dotsc, m-1$. The disk $D_m$ is also allowed to have interior punctures $z_1, \dotsc , z_l$, so that near $z_i$ the map $u$ is asymptotic to a cylinder over a  Reeb orbit $\gamma_i$ at $- \infty$, $i = 1, \dotsc, l$. (A \emph{Reeb orbit} is a periodic solution to the Reeb vector field.)  

The lower level consists of $J$-holomorphic maps $v_i \colon \C \to X$, $i = 1, \dotsc, l$,  from the punctured sphere and where $v_i$ maps a neighborhood of the puncture asymptotically to the Reeb orbit cylinder over the  Reeb orbit $\gamma_i$.

Let $\M_{A}^{\R \times V;X}(a, {\bf b})$, ${\bf b} = b_1 \cdots b_m$, denote the moduli space of such buildings, where $A$ denotes the homology class of the building. Then the differential of $\Al(V, \Lambda)$ is defined on generators by
\begin{equation}\label{eq:buildingdiff}
 \partial(a) = \sum_{\dim \M_{A}^{\R \times V;X}(a, {\bf b}) =1} |\M_{A}^{\R \times V;X}(a, {\bf b})|A {\bf b}
\end{equation}
where $|\M_{A}^{\R \times V;X}(a, {\bf b})|$ is the mod 2 count of $\R$-components in the moduli space, and the differential is extended to the whole of $\Al(V,\Lambda)$ by the Leibniz rule. For proofs that the homology of this DGA gives a Legendrian invariant we refer to \cite{sft, surg1,tobdisk}.

\subsection{Legendrian contact homology in 1--jet spaces}\label{sec:1jet}
Let $M$ be a smooth manifold of dimension $n$. Then the \emph{$1$-jet space of $M$}, $J^1(M) =  T^*M\times \R$, is a contact manifold with contact form $\alpha =  dz- vdu$, where $u$ are local coordinates on $M$, $v$ are cotangent coordinates and $z$ is the coordinate in the $\R$-direction.
The Reeb vector field is given by $\partial_z$.

In this case one can use the \emph{Lagrangian projection} 
\begin{equation*}
 \Pi_\C: J^1(M) \to T^*M
\end{equation*}
to study Legendrian submanifolds. 
The Legendrians are projected to exact, immersed Lagrangians of $T^*M$ under this projection, and the double points of $\Pi_\C(\Lambda)$ correspond to Reeb chords of $\Lambda$.  Moreover,  after a small Legendrian isotopy of $\Lambda$ we may assume that it is \emph{chord generic}, meaning that $\Pi_\C(\Lambda)$ is an immersion with transverse double points as the only intersections. This implies that if $\Lambda$ is closed, then the number of Reeb chords of $\Lambda$ is finite.

\subsubsection{Grading}
The grading of a Reeb chord of $\Lambda$ can be explicitly described as follows. Consider the  \emph{front projection}
\begin{equation*}
\Pi_F \colon J^1(M) \to M \times \R,
\end{equation*}
and  the  \emph{base projection} 
\begin{equation*}
\Pi \colon J^1(M) \to M.                                                                                                                                                                  \end{equation*} We will assume that $\Lambda$ is \emph{front generic}. We refer to [\cite{legsub}, Section 3.2] for a definition of this, but briefly this means that $\Pi|_\Lambda$ is an immersion outside a  co-dimension 1 singular set $\Sigma \subset \Lambda$, and that there is a subset $\Sigma' \subset \Sigma$ of codimension 1 so that the points in $\Sigma \setminus \Sigma'$ belong to a standard cusp singularity of the front projection. The points in $\Sigma\setminus \Sigma'$ will be called the \emph{cusp edge points} of the front of $\Lambda$.

Let $\Lambda_1, \dotsc, \Lambda_s$ be the connected components of $\Lambda$. 
For each component $\Lambda_j$ fix a point $q_j \in \Lambda_j$ so that $q_j$ does not project to a singularity under the front projection and so that it does not coincide with a Reeb chord start or end point.

For each pair $\Lambda_i, \Lambda_j$ such that there is a Reeb chord between them, pick one such chord $c_{ij}$, which we will call a \emph{connecting chord}. 
Let $c_{ij,\pm}$ be the start and end point of $c_{ij}$ so that $z(c_{ij,+}) > z(c_{ij,-})$. 
Suppose that $c_{ij,+} \in \Lambda_j, c_{ij,-} \in \Lambda_i$. 
Then there are locally defined functions $f_i,f_j \colon U \to \R$, where $U \subset M$ and $\Pi(c_{ij}) \in U$, so that a neighborhood of $c_{ij,-}$   in $\Lambda_i$  and of $c_{ij,+}$ in $\Lambda_j$ is given by 
\begin{equation*}
 \{(u,df_i(u), f_i(u)); u \in U\} \text{ and } \{(u,df_j(u), f_j(u)); u \in U\},
 \end{equation*}
respectively,  and $c_{ij}$ corresponds to a non-degenerate critical point of $f_j - f_i$. 
Let  
\begin{equation*}
I(c_{ij}) = \ind_{c_{ij}}(f_j - f_i)
\end{equation*}
be the Morse index of $f_j - f_i$ at this critical point. 

\begin{definition}
 Assume that $\Lambda \subset J^1(M)$ is a front generic Legendrian submanifold. A path $\gamma \subset \Lambda$ is \emph{admissible} if it intersects the singularities of the front projection of $\Lambda$ transversely at cusp edges.
\end{definition}

\begin{definition}\label{def:du}
If $\gamma \subset \Lambda$ is an oriented admissible path, let $D(\gamma)$ ($U(\gamma)$) be the number of cusp edges of $\Lambda $ that $\gamma$ transverses downwards (upwards) with respect to the $z$-coordinate.
\end{definition}

Now pick admissible paths $\tilde \gamma_{c_{ij,j}}\subset \Lambda_j$ from $p_j$ to $c_{ij,+}$ and $\tilde \gamma_{c_{ij,i}}\subset \Lambda_i$ from  $p_i$ to $c_{ij,-}$. 
Let 
\begin{equation}\label{eq:I1}
I_{ij} = D(\tilde \gamma_{c_{ij,j}}) - U(\tilde \gamma_{c_{ij,j}}) - D(\tilde \gamma_{c_{ij,i}}) + U(\tilde \gamma_{c_{ij,i}}) - I(c_{ij}),
\end{equation}
for $i \neq j$ and $c_{ij}$ going from $\Lambda_i$ to $\Lambda_j$, and let 
\begin{equation}\label{eq:I2}
I_{ji} = - I_{ij}, \qquad  I_{ii} = 0. 
\end{equation}    

If now $c$ is a Reeb chord of $\Lambda$ with $c_\pm \in \Lambda_{i\pm}$, pick admissible paths $\gamma_{c,\pm} \subset \Lambda_{i\pm}$ from $c_\pm$ to $q_{i\pm}$. These paths are called \emph{capping paths} for $c$.

\begin{definition}\label{def:grading}
The grading of a Reeb chord $c$ of $\Lambda$ is given by 
\begin{equation*}
|c| = D(\gamma_{c,+}) - U(\gamma_{c,+}) -D(\gamma_{c,-}) + U(\gamma_{c,-}) + I(c) +I_{i-i+}-1.
\end{equation*}
\end{definition}

\subsubsection{Differential}
The differential of $\Al(J^1(M),\Lambda)$ is defined by counting pseudo-holomorphic disks of the Legendrian $\Lambda$. This can be done in two different ways, either by counting disks in the cotangent bundle $T^*M$, with the disks having boundary on the Lagrangian $\Pi_\C(\Lambda)$, or by counting disks in the symplectization of $J^1(M)$, with the disks having boundary on the Lagrangian $\R \times \Lambda$. In \cite{Georgios} it is proven that for certain choices of almost complex structures these two different set-ups give the same count of elements mod 2, and in \cite{teckcob} this is proven to hold also with $\Z$-coefficients.

 We give the definition of the count in the cotangent bundle, and refer to \cite{ratsft,Georgios,teckcob} for the definition of the count in the symplectization. 

Let $J$ be an almost complex structure of $T^*M$, compatible with the standard symplectic structure. Let $D_{m+1}$ denote the punctured unit disk in $\C$ with $m+1$ punctures $p_0,\dotsc,p_m$ cyclically ordered along the boundary in the counterclockwise direction, starting at $p_0 =1$. Let $a, b_1,\dotsc,b_m$ be Lagrangian projections of Reeb chords.
\begin{definition}
We say that 
\begin{equation*}
u: (D_{m+1}, \partial D_{m+1}) \to (T^*M, \Pi_\C(\Lambda))
\end{equation*}
 is a \emph{$J$-holomorphic disk of $\Lambda$} with positive puncture $a$ and negative punctures \linebreak
 ${\bf b} =b_1 \cdots b_m$ if

 \begin{itemize} 
 \item $\dbar_J u := du + J du \circ i =0$,
 \item $u|_{\partial D_{m+1} \setminus\{p_0,\dotsc,p_m\}}$ has a continuous lift $\tilde u: {\partial D_{m+1} \setminus\{p_0,\dotsc,p_m\}} \to \Lambda$,
 \item $u(p_0)=a$, and $\tilde u$ makes a jump from lower to higher $z$-coordinate when passing through $p_0$ in the counterclockwise direction,
  \item $u(p_i)=b_i$, $i=1,\dotsc, m$, and $\tilde u$ makes a jump from higher to lower $z$-coordinate when passing through $p_i$ in the counterclockwise direction. 
\end{itemize}
\end{definition}
We let $\M(a, {\bf b}) = \M_{\Pi_\C(\Lambda)}(a, {\bf b})$ denote the moduli space of $J$-holomorphic disks of $\Lambda$ with positive puncture $a$ and negative punctures ${\bf b}$. We consider two disks in the moduli space to be equal if they only differ by a biholomorphic reparametrization of the domain. 

We define the differential of $\Al(J^1(M),\Lambda)$ to be given by
\begin{equation}\label{eq:thesum}
 \partial(a) = \sum_{\dim \M(a, {\bf b}) =0} |\M(a, {\bf b})|_R {\bf b}
\end{equation}
on generators $a$, and  extend it to the whole of $\Al(\Lambda)$ by the Leibniz rule. 
Here  $|\M(a, {\bf b})|_R \in R$ is the algebraic count of elements in the moduli space, which in the case of $R =\Z_2$ is given by the modulo 2 count.    

In \cite{PR} it is proven that the homology of $\Al(V,\Lambda)$ is a well-defined Legendrian invariant, that is, $\partial^2 =0$ and the homology  is invariant under Legendrian  isotopies.

\subsubsection{Morse flow trees}
Instead of using pseudo-holomorphic disks to define the differential, one can as well use \emph{Morse flow trees}. These are defined as follows. 

Let $\Lambda \subset J^1(M)$ be a chord generic Legendrian submanifold with \emph{simple front singularities}, meaning that the codimension 2 subset $\Sigma' \subset \Lambda$ where the singularities of the base projection do not consist of cusp singularities is empty. If $\dim \Lambda = 2$ we may also allow swallow tail singularities, see [\cite{trees}, Section 2.2.A].

 Away from the singular set $\Sigma$, the pre-image of an open set $U \subset M$ under the base projection $\Pi$ is given by the multi-1-jet lift of locally defined functions 
 \begin{equation*}
f_1,\dotsc,f_k : U \to \R, \quad U \subset M,                                                                                                                                                           \end{equation*}
that is,
\begin{equation*}
 \Pi^{-1}(U)|_{\Lambda} = \bigcup_{i=1}^k  \{(x, df_i(x), f_i(x));\, x \in U\}.
\end{equation*}
%
%
These locally defining functions of $\Lambda$ are used to build the Morse flow trees. More precisely, after having fixed a metric on $M$, these trees are defined as follows.

\begin{definition}
A \emph{Morse flow tree} is an immersed tree $\Gamma$ in $M$  satisfying the following conditions.

\begin{itemize}
 
 \item The tree is rooted and oriented away from the root. The root is $1$- or $2$-valent.
\item Each edge $\gamma$ of $\Gamma$ is a solution curve of some local function difference: 
\begin{equation*}
 \dot \gamma = \dot \gamma_{ij}(t) = - \nabla (f_i- f_j)(\gamma_{ij}(t)), 
\end{equation*}
where $f_i> f_j$ are locally defining functions of $\Lambda$. 
 \item The edge $\gamma_{ij}$ is given the orientation of 
 $- \nabla(f_i -f_j)$.

 \item The cotangent lift of $\Gamma$ gives an oriented closed curve in $\Pi_\C(\Lambda)$, in the following way. Each edge $\gamma_{ij}$ has two cotangent lifts 
 \begin{equation*}
 \hat \gamma_{ij,k} = \{(x, df_k(x)); \, x \in \gamma_{ij} \} \subset \Pi_\C(\Lambda), \quad k=i,j.
\end{equation*}
 If we give $\hat \gamma_{ij,i}$ the orientation of $\gamma_{ij}$, and $\hat \gamma_{ij,j}$  the negative orientation of $\gamma_{ij}$, then  
the union of all the lifted edges of $\Gamma$ are required to patch together to give  a closed curve in $\Pi_\C(\Lambda) \subset T^*M$. 
 
 \item The vertices of $\Gamma$ have valence at most $3$, and
are of the following form.
\begin{itemize}
 \item $1$-valent \emph{punctures}, which are critical points of the corresponding local function difference,
 \item $2$-valent \emph{punctures}, which are critical points of the corresponding local function difference,
 \item $3$-valent \emph{$Y_0$-vertices}, where flow lines $\gamma_{ij}, \gamma_{jk}$, $\gamma_{ik}$ meet,
 \item $3$-valent \emph{$Y_1$-vertices}, similar to $Y_0$-vertices but contained in $\Pi(\Sigma)$,
 \item $2$-valent \emph{switch-vertices}, contained in $\Pi(\Sigma)$, with corresponding flow lines which are tangent to $\Pi(\Sigma)$ at the vertex,
 \item $1$-valent \emph{end-vertices}, contained in $\Pi(\Sigma)$, with corresponding flow lines which are transverse to $\Pi(\Sigma)$ at the vertex.
\end{itemize}
\item The root of the tree $\Gamma$ is required to be a puncture, and is called the \emph{positive puncture} of the tree. All other punctures are called \emph{negative}. 
\end{itemize}
\end{definition}

Since a puncture $p$ is a critical point of a local function difference, we have stable and unstable manifolds associated to $p$. These we denote by $W^s(p)$ and $W^u(p)$, respectively.

The dimension of a Morse flow tree $\Gamma$ with positive puncture $a$ and negative punctures $b_1,\dotsc,b_m$ can be computed using data from the tree, and is given by
\begin{equation*}
 \dim(\Gamma) = 2 + \dim W^u(a) + \sum_{j=1}^m (\dim W^s(b_j) - n + 1) + e(\Gamma) - s(\Gamma) - Y_1(\Gamma),
\end{equation*}
where $e(\Gamma)$, $s(\Gamma)$, $Y_1(\Gamma)$ is the number of end-, switch- and $Y_1$-vertices of $\Gamma$.

\begin{definition}
 A \emph{rigid Morse flow tree of $\Lambda$} is a Morse flow tree of dimension $0$ which is transversely cut out from the space of flow trees.
\end{definition}

In \cite{trees} it is proven that one can define the differential by counting rigid Morse flow trees of $\Lambda$ instead of counting rigid pseudo-holomorphic disks. In some situations this gives an easier way to understand $\Al(J^1(M), \Lambda)$, since this avoids solving a $\dbar$-equation, which is a non-linear PDE.

\section{The Chekanov-Eliashberg DGA in a subcritical Weinstein manifold}\label{sec:dga}
Let $\Lambda$ be a Legendrian submanifold of a contact manifold $V$ which is the boundary of a subcritical Weinstein manifold $X$ of dimension $2n$, $n>2$. Assume that $c_1(X) = 0$. 
In this section we describe the Chekanov-Eliashberg DGA of $\Lambda$, $\Al(V,\Lambda)$, in terms of sub-DGAs which can be computed from Legendrians in 1--jet spaces. To do this we need to put some additional assumptions on $V$ and $\Lambda$.

\subsection{Preliminary assumptions}\label{sec:assumptions}
To simplify notation we will assume that  $X$ only has one subcritical handle $\Ha^k$ attached. This  easily generalizes to the case of having several subcritical handles attached along isotropic spheres where no attaching spheres of the subcritical handles passes through any other subcritical handle. We also assume that $k<n-1$. The more general situation will be dealt with in a forthcoming paper. 

So let $\Upsilon \subset S^{2n-1} = \partial B^{2n}$ be the isotropic sphere along which the handle $\Ha^k$ is attached. 
We will assume that there is a Darboux ball $B_A \subset S^{2n-1}$ of radius $A$ containing  the attaching region $N(\Upsilon)$ of the handle. 
This means that we have a contactomorphism
\begin{equation*}
\phi \colon B_A \to D_A = \{(x,y,z) \in J^1(\R^{n-1}); |x|^2 + |y|^2 +z^2 \leq A^2\}, \quad \phi^* (dz - ydx) = \alpha_{S^{2n-1}},
\end{equation*}
where $\alpha_{S^{2n-1}}$ is the standard contact structure on $S^{2n-1}$.
Thus we can consider the handle attachment as being performed in $D_A \subset J^1(\R^{n-1})$ instead. 
Let $D_A^H$ denote the resulting surgered disk, and let $B_A^H = \phi^{-1}(D_A^H)$, where $\phi$ is extended by the identity over the handle. 

\begin{lemma}\label{lma:quasi}
Let $n \geq 2$ and let $\Lambda \subset V^{2n-1}$ be a Legendrian submanifold such that $\Lambda \subset B_A^H$ and assume that $B_A$ does not intersect any coordinate subspaces $\{z_i =0\}$ of $S^{2n-1} = \{|z|^2 = 1 \} \subset \C^n$. 
If $A$ is sufficiently small, then $\Al(V,\Lambda) \underset{\mathrm{quasi}}{\simeq} \Al(D_A^H, \phi(\Lambda))$.
\end{lemma}
\begin{proof}
This follows similarly to [\cite{en}, Lemma 5.10].
\end{proof}

Assuming that $V$ and $\Lambda$ satisfy the requirements of this lemma we will consider $\Lambda$ as a subset of $D_A^H$ from now on, dropping the map $\phi$ to simplify notation.

We will need some further assumptions on $\Lambda$ to be able to describe $\Al(D_A^H, \Lambda)$ in terms of sub-DGAs of Legendrians in 1--jet spaces. First, we need to assume that there is an $a < A$ such that the attaching sphere of the handle $\Ha^k$ is contained in $D_A \setminus D_a$ and that 
\begin{equation}
\Lambda \cap \left(D_A^H \setminus \left(\Ha^k_+ \cup N(\Upsilon) \right) \right) \subset D_a. 
\end{equation}
Let 
\begin{align*}
  T^*_{\rho_1}(S^{k-1}) &= \{(u,v) \in \R^{2k};|u| = 1, u \cdot v = 0, |v|^2 < \rho_1\}, \\
     D_{\rho_2}^{2n-2k}   &= \{(s,t) \subset \R^{2n-2k}; |s|^2 + |t|^2 < \rho_2\}, \\ 
     I_{\rho_3}   &= \{r \in \R; |r| < \rho_3\},
 \end{align*}
 and consider $T^*_{\rho_1}(S^{k-1})\times D_{\rho_2}^{2n-2k} \times I_{\rho_3}$ 
with contact form $\alpha_N = dr - vdu - tds$. Then there is a contactomorphism
\begin{equation*}
\psi \colon N(\Upsilon) \to  T^*_{\rho_1}(S^{k-1})\times D_{\rho_2}^{2n-2k} \times I_{\rho_3}
\end{equation*}
for $\rho_1, \rho_2, \rho_3$ sufficiently small, which maps $\Upsilon$ to the zero-section of $T^*_{\rho_1}(S^{k-1})$.

Let 
\begin{equation*}
D^k = \{(x,y,p,q) \in \Ha^k;  x=y=p= 0\}
\end{equation*}
be the core of the handle $\Ha^k$ and let
\begin{equation*}
 C^{2n-k} = \{(x,y,p,q) \in \Ha^k; q = 0\}
\end{equation*}
be the cocore. Assume that 
\begin{equation*}
 \Lambda \cap \partial C^{2n-k} = \{(x,y,p,q) \in \Ha^k_+; (x,y) \in \Lambda_\sub, p= q = 0\}
\end{equation*}
where $\Lambda_{\sub} \subset S^{2n-2k-1}$ is a Legendrian submanifold with respect to the standard contact structure.   We also assume that $\Lambda \cap \Ha_+^{k}$ is of the form 
  \begin{equation}\label{eq:subhandle}
  \{(r(q)x, r(q)y, p , q) \in \Ha_+^k ; (x,y) \in \Lambda_{\sub}, p = 0\},
  \end{equation}
  where $r \colon D^k \to \R_{\geq 0} $ is a Morse function with exactly one critical point, located at the origin and of index $0$. We also assume that  $\Lambda \cap \Ha_+^k$ is contained in a 1--jet neighborhood of the \emph{standard Legendrian cylinder} $\Lambda_\st \subset \Ha_+^k$, given by 
\begin{align*}
 \Lambda_\st &= \{(x,y,p,q) \in \Ha_+^k; y = p = 0 \} = \{(x,y,p,q) \in \Ha_+^k;\sum_{i=1}^{n-k}x_i^2 = 2(\delta^2 + \sum_{i=1}^k q_i^2)\} \\
 &\simeq S^{n-k-1} \times D^k. 
\end{align*}

By identifying $\Ha^k_+$ with $\Ha^k_-$ using the Liouville flow in $\Ha^k$ and then identifying a region of $\Ha^k_-$ with the attaching region, 
we see that we might assume the projection of $\psi(\Lambda \cap N(\Upsilon))$ to  $T^*_{\rho_1}(S^{k-1})\times I_{\rho_3}$ to coincide with the zero section of $T^*_{\rho_1}(S^{k-1})$ and the projection of $\psi(\Lambda \cap N(\Upsilon_i^K))$ to $D_{\rho_3}^{2n-2k} \setminus D_{\rho_3'}^{2n-2k} \simeq S^{2n-2k-1} \times [\rho_3',\rho_3]$ to coincide with 
\begin{equation}\label{eq:subattach}
 (\rho^2 \Lambda_{\sub},\rho), \quad \rho \in [\rho_3',\rho_3]
  \end{equation} 
  for some $\rho_3'> 0$. See Section \ref{sec:setup}.

From these assumptions it follows that we can cover $\Lambda$ by charts given by Legendrians in $D_a \subset J^1(\R^{n-1})$ and in $J^1(S^{n-k-1} \times D^k)$. In Section \ref{sec:dippings} we will describe an isotopy of $\Lambda$ allowing us to describe $\Al(D_A^H, \Lambda)$ in terms of subalgebras, where  each subalgebra can be computed in one of the 1--jet spaces just described.

\subsection{The differential of $\Al(V, \Lambda)$}
The differential of $\Al(V, \Lambda)$ is a priori given by a count of pseudo-holomorphic curves anchored in $X$ as in Section \ref{sec:fillable}. However, by similar arguments as in \cite{en} and also by the work in \cite{tobdisk} it follows that it is enough to consider pseudo-holomorphic disks in the symplectization of $V$. In this subsection we will investigate these disks further.

Recall that the Chekanov-Eliashberg algebra $\Al(V,\Lambda)$ is generated by the Reeb cords of $\Lambda$. 
By Lemmas \ref{lma:quasi}, \ref{lma:subalg1} and \ref{lma:subalgk} it is enough to consider the following Reeb chords of $\Lambda$.
\begin{description}
\item[Diagram chords] The Reeb chords $a_1, \dotsc, a_{\tilde m}$ of $\Lambda \cap D_a \subset (J^1(\R^{n-1}), dz - y dx)$. 
\item[Handle chords] Let $b_1, \dotsc , b_m$ be the Reeb chords of $\Lambda_\sub$ seen as a Legendrian submanifold of  $J^1(\R^{n-k-1})$, using Lemma \ref{lma:quasi}. Let $b_1[h], \dotsc , b_m[h]$ be the copies of these chords located at $ \Lambda_\sub \times \{0,0\} \subset \Ha^k_+$. 
\end{description}

\begin{remark}
 Here we see the the reason why we have to exclude the $k =(n-1)$-case for the moment, since Lemma \ref{lma:quasi} does not hold for $\Lambda_\sub$ for this $k$.
\end{remark}

We would like to be able to make a similar partition of the pseudo--holomorphic curves which contribute to the differential. To be able to do this, we isotope $\Lambda$ in the attaching region, to introduce a high-dimensional counterpart of the dippings from \cite{sabloffdipp}.

Recall that we assume $\Lambda$ to be of the form \eqref{eq:subhandle} and \eqref{eq:subattach} in $\Ha^k_+$ and $N(\Upsilon)$, respectively. In Section \ref{sec:setup} we prove that we have a sub-algebra $\Al(J^1(\R^{n-k-1}), \Lambda_\sub)$ at the minimum $q=0$ in the handle. However, we might have pseudo-holomorphic disks with positive punctures at diagram chords traveling into the handles. The dipping procedure will help us to get control over these disks.

\subsubsection{Dippings}\label{sec:dippings}
Let $f \colon [\rho_3', \rho_3] \to \R$ be a positive Morse function which coincides with $\rho^2$ for $\rho < \epsilon_1'$ and $\rho > \epsilon_1$ for some $\rho_3' < \epsilon_1' < \epsilon_1 < \rho_3$, and which has one maximum at $\rho = p_1$, one minimum at $\rho = p_2$ for some  $\epsilon_1' < p_1 < p_2 < \epsilon_1$, and no other critical points. Assume that $f(p_1) = p_1^2 + \delta_1$, $f(p_2) = p_2^2 - \delta_2$, where $\epsilon_1, \epsilon_1', p_1, p_2, \delta_1, \delta_2$ are the \emph{dipping parameters} and are to be chosen.

Now we Legendrian isotope 
\begin{equation*}
 \Lambda \cap N(\Upsilon) \simeq \{(\sigma, \rho^2u, \rho^2v, \rho) ; \sigma \in \Upsilon, (u,v) \in \Lambda_\sub, \rho \in [\rho_3', \rho_3]\}
\end{equation*}
to the Legendrian 
\begin{equation*}
\{(\sigma, f(\rho)u, f(\rho)v, \rho) ; \sigma \in \Upsilon, (u,v) \in \Lambda_\sub, \rho \in [\rho_3', \rho_3]\}.
\end{equation*}
To simplify notation we continue to denote the isotoped Legendrian by $\Lambda$. 

If $k>1$ this gives us a Morse--Bott situation where we for each Reeb chord of $\Lambda_\sub$ get one $S^{k-1}$-family of Reeb chords for $\rho = p_1$ and another $S^{k-1}$-family for $\rho = p_2$. To avoid this situation let $g \colon S^{k-1} \to \R$ be a positive Morse function with one maximum at $\sigma_1 \in S^{k-1}$, one minimum at $\sigma_2 \in S^{k-1}$ and no other critical points. Legendrian isotope $\Lambda \cap N(\Upsilon)$ to the Legendrian 
\begin{equation*}
\{(\sigma, (1+\chi(\rho)g(\sigma))f(\rho)u, (1+\chi(\rho)g(\sigma))f(\rho)v, \rho) ; \sigma \in \Upsilon, (u,v) \in \Lambda_\sub, \rho \in [\rho_3', \rho_3]\},
\end{equation*}
where $\chi \colon [\rho_3', \rho_3] \to \R$ is a bump function as in Figure \ref{fig:bump}. We continue to denote the isotoped Legendrian by $\Lambda$. 

By choosing the height $h$ of the bump function $\chi$ small enough we can ensure that we get exactly four critical points for the function
$(1+\chi(\rho)g(\sigma))f(\rho)$ on $S^{k-1} \times [\rho_3', \rho_3]$, as in Figure \ref{fig:morsebott}. That is, we get critical points 
\begin{align*}
 m_1 &= (\sigma_1, p_1) \qquad \text{ of index } k \\
 s_1 &= (\sigma_2, p_1) \qquad \text{ of index } 1 \\
 s_2 &=  (\sigma_1, p_2) \qquad \text{ of index } k-1  \\
 m_2 &= (\sigma_2, p_2) \qquad \text{ of index } 0.
\end{align*}

  \begin{figure}[ht]
 \includegraphics[height=10cm, width=12cm]{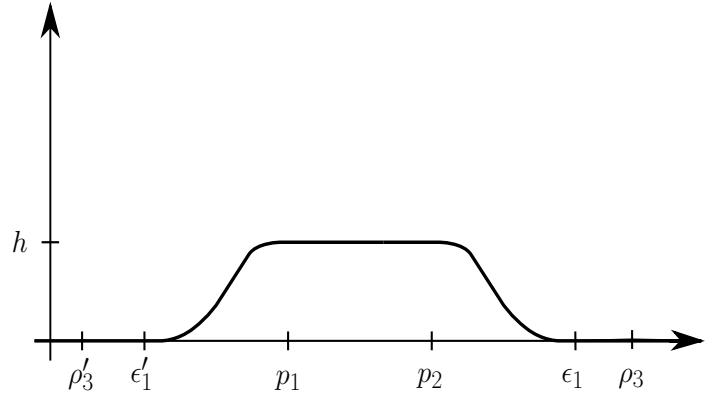}
 \vspace{-3cm}
\caption{The bump function $\chi$.}
 \label{fig:bump}
 \end{figure}
 
   \begin{figure}[ht]
 \centering
\includegraphics[height=10cm, width=13cm]{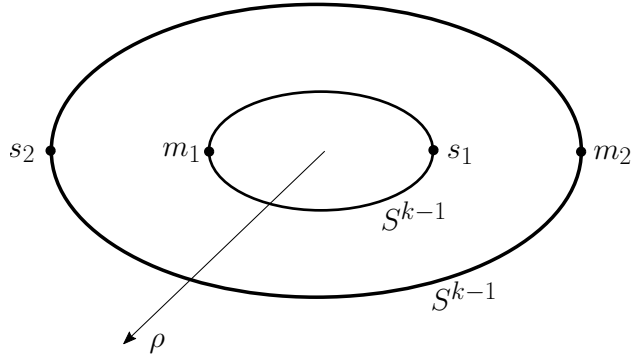}
 \vspace{-4cm}
  \caption{The critical points of the dipping function.}
 \label{fig:morsebott}
 \end{figure}
 
Let 
\begin{equation*}
 A_d = \{(\sigma, u , v, \rho) \in N(\Upsilon); \epsilon_1' < \rho < \epsilon_1\}
\end{equation*}
be the region where the dipping is performed and consider $\Al(A_d, \Lambda \cap A_d)$.
If $1<k<n-1$ this algebra has generators $b_i[m_1], b_i[s_1], b_i[s_2], b_i[m_2]$, $i=1, \dotsc, m$, where $b_1, \dotsc, b_m$ are the generators for $\Al(J^1(\R^{n-k-1}), \Lambda_\sub)$ and where  $b_i[p]$ is the Reeb chord $b_i$ of the copy of $\Lambda_\sub$ located at the critical point $p \in \{m_1, s_1, m_2, s_2\}$.  These Reeb chords are the \emph{dipping chords} of $\Lambda$.

If $k= 1$ we do not need the function $g$, since now $S^{k-1}$ is given by two points $\sigma_0$ and $\sigma_0'$.  It is thus enough to consider the function $f$ defined on $S^{k-1} \times [\rho_3', \rho_3] \simeq D_1 \sqcup D_1$, and we get four critical points

\begin{align*}
 m_1 &= (\sigma_0, p_1) \qquad \text{ of index } 1 \\
 s_1 &= (\sigma_0', p_1) \qquad \text{ of index } 1 \\
 m_2 &= (\sigma_0, p_2) \qquad \text{ of index } 0 \\
 s_2 &= (\sigma_0', p_2) \qquad \text{ of index } 0.
\end{align*}
Hence we get generators $b_i[m_1], b_i[s_1], b_i[s_2], b_i[m_2]$, $i=1, \dotsc, m$ of $\Al(A_d, \Lambda \cap A_d)$, using similar notation as in the $k>1$ case.

We will assume that the dipping region intersects $D_a$ along $\rho = \epsilon_2$, where $\epsilon_2 \in (p_1, p_2)$, so that the critical points $m_2, s_2$ are contained in $D_a$, but not the points $m_1, s_1$.

\subsubsection{Gradings}\label{sec:gradings}
We define gradings of the Reeb chords $b_1, \dotsc, b_m$  of $\Lambda_\sub \subset J^1(\R^{n-k-1})$ as in Section \ref{sec:1jet}.  That is, for 
$\Lambda_1', \dotsc, \Lambda_{s'}'$  the connected components of $\Lambda_\sub$ we choose marked points $q_i' \subset \Lambda_i'$, $i = 1, \dotsc, s'$, and connecting  chords $c_{ij}'$ from $\Lambda_i'$ to $\Lambda_j'$, $i \neq j$, as described in that section, together with admissible paths  $\tilde \gamma_{c_{ij,i}'}'\subset \Lambda_i'$, $\tilde \gamma_{c_{ij,j}'}'\subset \Lambda_j'$ from $q_i'$ to $c_{ij,-}'$  and from $q_j'$ to $c_{ij,+}'$, respectively. 
For each Reeb chord $b$ we also choose capping paths $\gamma_{b \pm}' \subset \Lambda_{\sub}$. 
With this data we can define a grading $|b_i|_\sub$ for $i = 1, \dotsc, m$, given by Definition \ref{def:grading}. 
See Figure \ref{fig:connectingchord}.

To define gradings for the chords $a_1, \dotsc, a_{\tilde m}$, $b_1[p], \dotsc, b_m[p]$, $p = m_1, m_2, s_1,s_2, h$, we proceed as follows. 
Let $\Lambda_1, \dotsc, \Lambda_s$ be the connected components of $\Lambda$. 
For each component $\Lambda_j$ fix a point $q_j \in \Lambda_j \cap D_a$ so that $q_j$ does not project to a singularity under the front projection and so that it does not coincide with a Reeb chord start or end point. Also, for 
each pair $\Lambda_i, \Lambda_j$ such that there is a Reeb chord between them, pick one such chord $c_{ij} \in D_a$ as connecting chord. (Note that this is possible since we assume the dipping region to intersect $D_a$.) 

\begin{definition}
 We say that a path $\gamma \subset \Lambda$ is \emph{handle admissible} if $\gamma \cap D_a$ is admissible and $ \gamma$ has constant projection to $\Lambda_\sub$ in $\Lambda \setminus (\Lambda \cap D_a) \simeq D^k \times \Lambda_\sub$. 
\end{definition}
\begin{definition}
 If $\gamma \subset \Lambda$ is handle admissible we let $D(\gamma) = D(\gamma \cap D_a)$, $U(\gamma) = U(\gamma \cap D_a)$, where $ D(\gamma \cap D_a), U(\gamma \cap D_a)$ is computed as in Definition \ref{def:du}.
\end{definition}

Now choose handle admissible paths as follows. 
\begin{itemize}
 \item For each connecting chord $c_{ij}$ choose paths $\tilde \gamma_{c_{ij,j}}\subset \Lambda_j$ from $p_j$ to $c_{ij,+}$ and $\tilde \gamma_{c_{ij,i}}\subset \Lambda_i$ from  $p_i$ to $c_{ij,-}$. \item For each diagram chord $ a = a_1, \dotsc, a_{\tilde m}$ with $a_\pm \in \Lambda_{l \pm}$ choose capping paths $\gamma_{a\pm} \subset \Lambda_{l \pm}$ from $a_\pm$ to $q_{l\pm}$. 
 \end{itemize}
 
 With these choices is it now possible to define gradings of $a_1, \dotsc, a_{\tilde m}$ as in Definition \ref{def:grading}.
 
To define gradings of the handle and dipping chords we use the following results.
\begin{lemma}\label{lma:gradingb}
 There is a choice of capping paths for $b_1[p], \dotsc, b_m[p]$, $p = m_1, s_1, m_2, s_2,h,$ and a function $K \colon \{1, \dotsc, s'\} \times \{1, \dotsc, s'\} \to \Z$ so that
 \begin{align*}
  |b_i[m_1]| &= |b_i|_\sub + K(i-, i+) + k \\
  |b_i[s_1]| &= |b_i|_\sub + K(i-, i+) + 1 \\
  |b_i[m_2]| &= |b_i|_\sub + K(i-, i+)  \\
  |b_i[s_2]| &= |b_i|_\sub + K(i-, i+) + k-1 \\
  |b_i[h]| &= |b_i|_\sub + K(i-, i+),  
 \end{align*}
$i = 1,\dotsc , m$, and where $i-, i+ \in \{1, \dotsc, s'\}$ satisfies $b_{i, \pm} \in \Lambda_{i\pm}'$. 
\end{lemma}

\begin{proof} 
To choose the capping paths, we first pick handle admissible paths as follows.
\begin{itemize}
 \item For each component $\Lambda_j'$ of $\Lambda_\sub$ let $l$ be such that $\Lambda_j' \subset \Lambda_l$ and let $\tilde \gamma_{jl} \subset \Lambda_l$ be a handle admissible path from $\{m_2\} \times \{q_j'\}\in A_d \cap D_a$ to $q_l$. 
 \item For each connecting Reeb chord $c_{ij}'$ of $\Lambda_\sub$, let $l\pm$ be such that $\Lambda_i'\subset \Lambda_{l-}, \Lambda_j' \subset \Lambda_{l+}$, and choose handle admissible paths $\tilde \gamma_{il-,\con} \subset \Lambda_{l-}, \tilde \gamma_{jl+,\con} \subset \Lambda_{l+}$ from $\{m_2\} \times \{c_{ij,-}'\}$ to $q_{l-}$ and from $\{m_2\} \times \{c_{ij,+}'\}$ to $q_{l+}$, respectively.

 \end{itemize}
 
  Let $\gamma_1*\gamma_2$ be the concatenation of the paths $\gamma_1$ and $\gamma_2$. To simplify notation, if $\gamma \subset \Lambda_\sub$ is a path we continue to write $\gamma$ for the copy $\{m_2\}\times \gamma$ of $\gamma$ in $\{m_2\} \times \Lambda_{sub}$.

To define the capping path $\gamma_{b\pm}$ for the Reeb chord $b = b_i[m_2]$, $i = 1, \dotsc, m$, assume that $b_{i,\pm} \in \Lambda_{i\pm}' \subset \Lambda_{l\pm}$. 
 We get the following cases, see Figures \ref{fig:case1}, \ref{fig:case2} and \ref{fig:case3}.
 \begin{align*}
  \gamma_{b\pm} =
  \begin{cases}
   \gamma_{b_i\pm}' * \tilde \gamma_{i\pm l\pm}, & i-= i+, \\
   \gamma_{b_i\pm}' *\tilde \gamma_{c_{i-i+,i\pm}'}'* \tilde \gamma_{i\pm l\pm,\con}, & i-\neq i+, \text{ the connecting chord of }\Lambda_{i-}', \Lambda_{i+}' \text{equals  } c_{i-i+}', \\
   \gamma_{b_i\pm}' *\tilde \gamma_{c_{i+i-,i\pm}'}'* \tilde \gamma_{i\pm l\pm,\con}, & i-\neq i+, \text{ the connecting chord of }\Lambda_{i-}', \Lambda_{i+}' \text{equals  } c_{i+i-}'.
  \end{cases}
 \end{align*}

 To define capping paths for the chords $b_i[p]$, $p = s_1,m_1,s_2,h$ we take the capping path of $b_i[m_2]$ and just extend it in a handle admissible way for $i = 1, \dotsc, m$ so that the extended parts do not intersect any cusps.

 Write $I'(c_{ij}')$ for the Morse index of $c_{ij}'$ regarded as a Reeb chord of $\Lambda_\sub$, and $I(c_{ij})$ for the Morse index of $c_{ij}$ regarded as a Reeb chord of $\Lambda$.
 
Now let  $K \colon \{1, \dotsc, s'\} \times \{1, \dotsc, s'\} \to \Z$ be given by 
 \begin{equation*}
    K(i-,i+) = 0,  \qquad i-= i+,
 \end{equation*}
 \begin{align*}
  K(i-,i+) =
   D(\tilde \gamma_{i+ l+,\con}) - U(\tilde \gamma_{i+ l+,\con}) - D(\tilde \gamma_{i- l-,\con}) +U(\tilde \gamma_{i- l-,\con}) + I'(c_{i-i+}) + I_{l-l+},\\ i-\neq i+,  \text{ the connecting chord of }\Lambda_{i-}', \Lambda_{i+}' \text{equals  } c_{i-i+}', 
 \end{align*}
 \begin{align*}
  K(i-,i+) =
   D(\tilde \gamma_{i+ l+,\con}) - U(\tilde \gamma_{i+ l+,\con}) - D(\tilde \gamma_{i- l-,\con}) + U(\tilde \gamma_{i- l-,\con}) - I'(c_{i+i-}) +  I_{l-l+},\\ i-\neq i+,  \text{ the connecting chord of }\Lambda_{i-}', \Lambda_{i+}' \text{equals  } c_{i+i-}', 
 \end{align*}
 where $I_{l-l+}$ is given by \eqref{eq:I1} and \eqref{eq:I2} and computed with respect to the connecting chords in $\Lambda$.   
 
 Comparing with Figures \ref{fig:case1}, \ref{fig:case2} and \ref{fig:case3} it is clear that the lemma follows. 
 \end{proof}

 \begin{figure}[ht]
 \includegraphics[height=10cm, width=13cm]{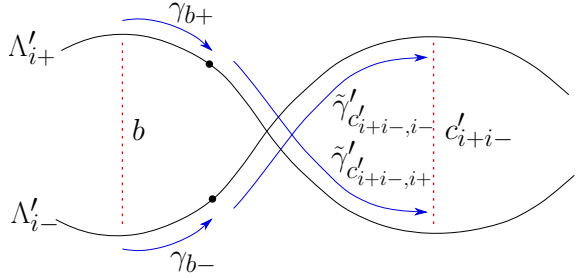}
 \vspace{-5cm}
 \caption{The choice of admissible paths for $\Lambda_\sub$.}
 \label{fig:connectingchord}
 \end{figure}

 \begin{figure}[ht]
 \includegraphics[height=10cm, width=13cm]{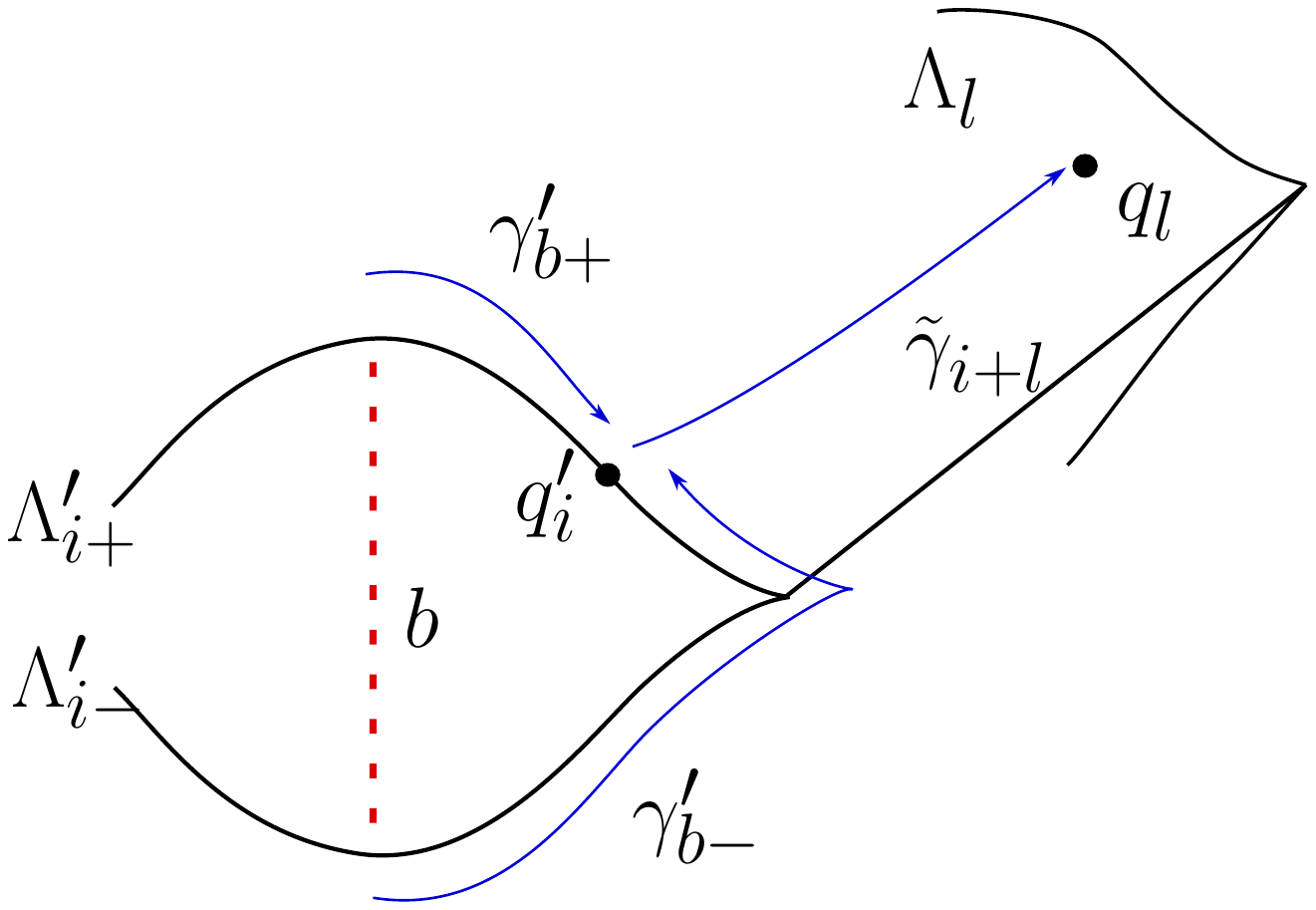}
 \vspace{-5cm}
 \caption{The choice of capping paths for $b[m_2]$ in $\Lambda$ in the case when when $i+ = i-$. }
 \label{fig:case1}
  \end{figure}

 \begin{figure}[ht]
 \includegraphics[height=10cm, width=13cm]{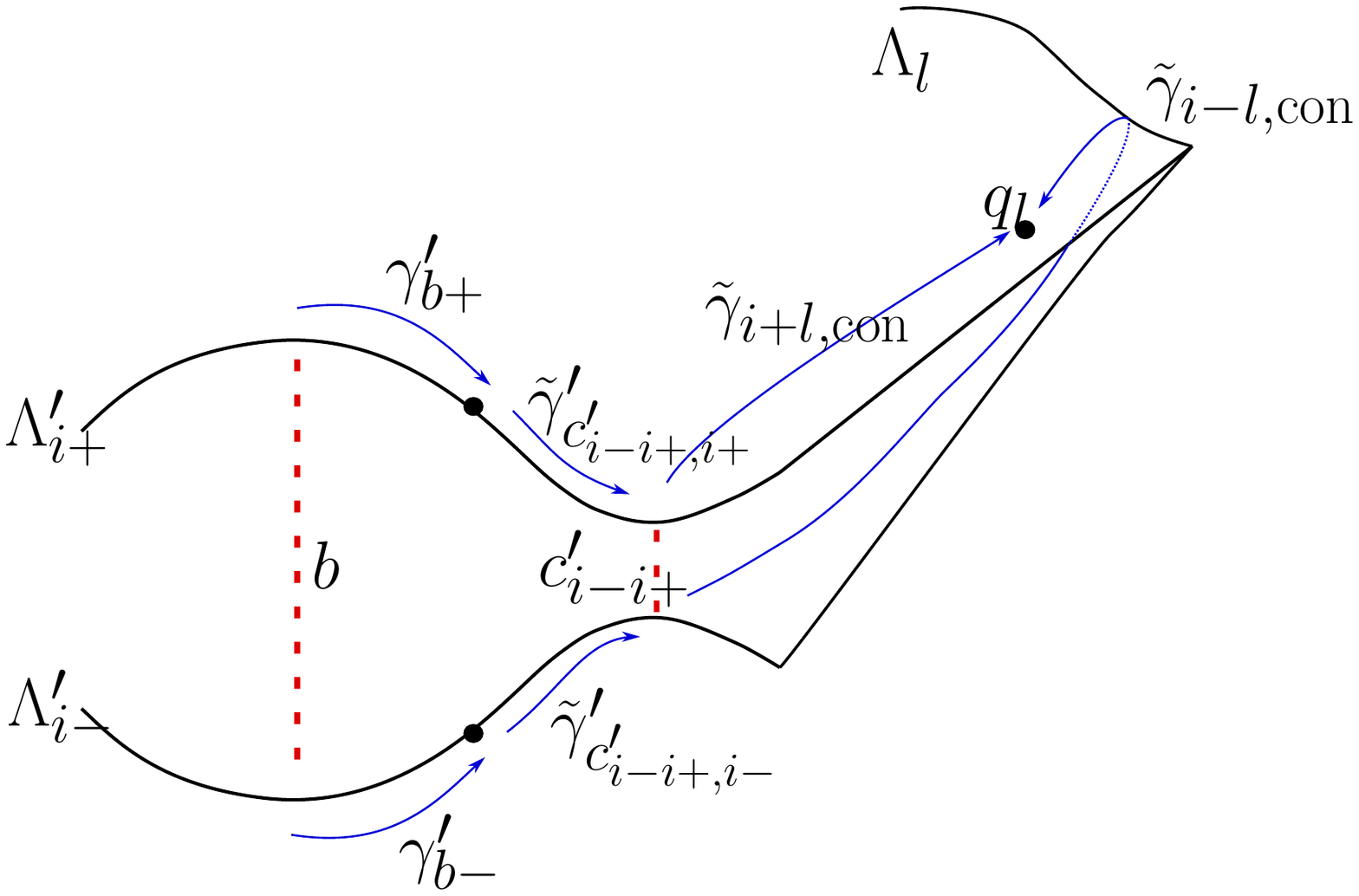}
 \vspace{-3cm}
 \caption{The choice of capping paths for $b[m_2]$ in $\Lambda$ in the case when when $i+ \neq i-$, $l+ = l-$ .}
 \label{fig:case2}
  \end{figure}

 \begin{figure}[ht]
 \vspace{-3cm}
\includegraphics[height=10cm, width=13cm]{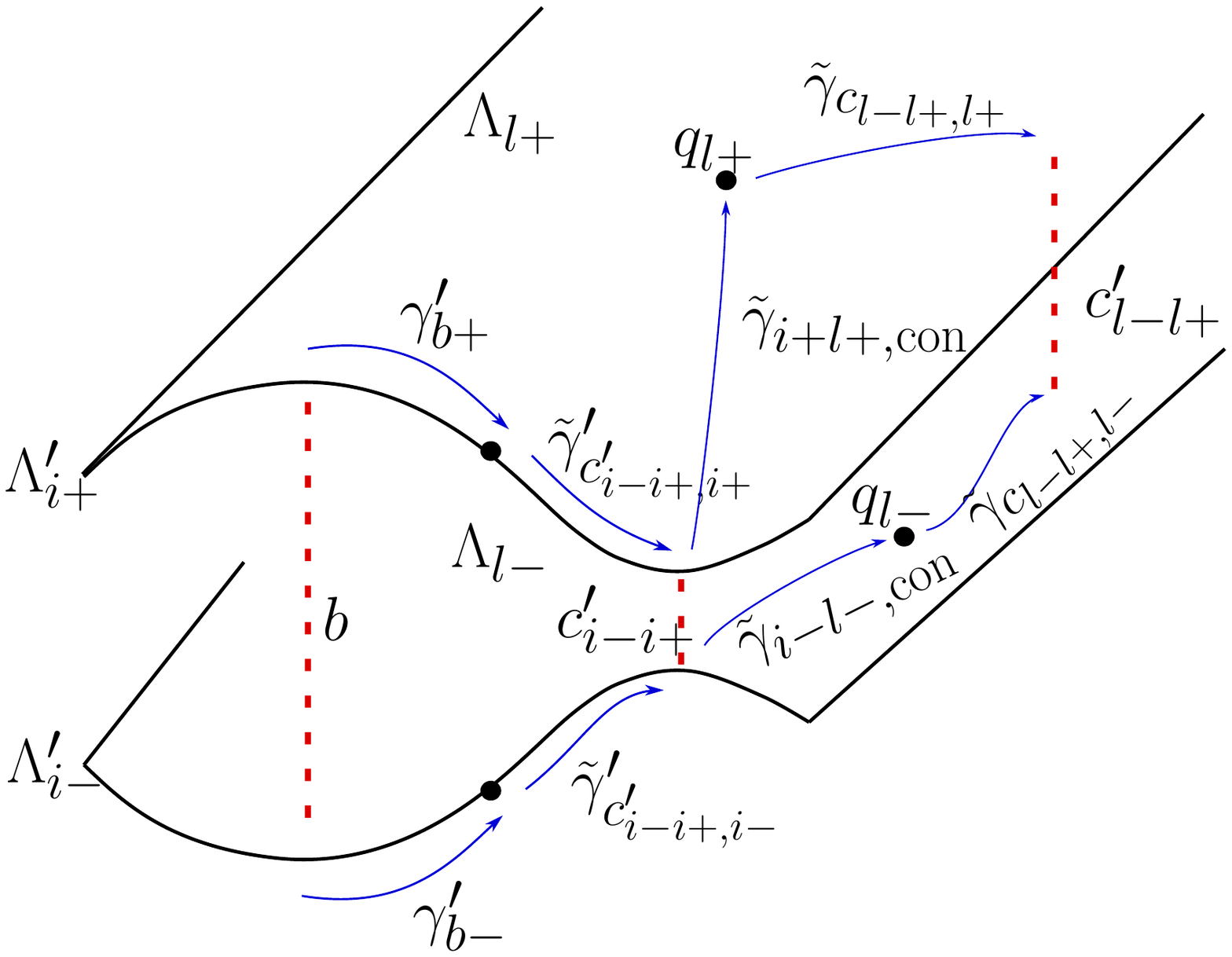}
 \vspace{1cm}
\caption{The choice of capping paths for $b[m_2]$ in $\Lambda$ in the case when  $i+ \neq i-$, $l+ \neq l-$. }
 \label{fig:case3}
 \end{figure}

\subsubsection{Almost complex structure}
To choose an almost complex structure for $\R_t \times D_A^H$ we proceed as in [\cite{en}, Section 5.2.A]. 
That is, in $D_a \subset J^1(\R^{n-1}) \simeq \C^{n-1} \times \R_z$  we choose the standard complex structure on $\C^{n-1}$ and then extend it to the whole space by requiring that $J(\partial_t) = \partial_z$.

In the handle we choose an almost complex structure which is standard in the 1--jet neighborhood of $\Lambda_\st$ where we assume that $\Lambda \cap \Ha^k$ is contained. 
This means that for $J^1(\Lambda_\st) \simeq T^*\Lambda_\st \times \R_z$ we assume it to be given as in [\cite{trees}, Section 4.3] in $T^*\Lambda_\st$, mapping the vertical subbundle of $T^*\Lambda_\st$ to the horizontal subbundle, where these subbundles are defined using some metric connection. 
Again we extend it to a cylindrical almost complex structure by setting $J(\partial_t) = \partial_z$. 
Extend this to a cylindrical almost complex structure in the rest of the handle. 

Assuming that $\Lambda_\st \cap D_a \subset N(\Upsilon)$ is contained in a real plane $(\R^{n-1} \times \{y_0\} \times \{z_0\})\cap D_a$, the almost complex structure defined in the handle will coincide with the almost complex structure in $\R \times D_a$ in the $1$--jet neighborhood of $\Lambda_\st$, assuming this is small enough, and we can interpolate the almost complex structures outside this neighborhood to give a cylindrical almost complex structure defined over the whole of $\R \times D_A^H$.  

\subsubsection{Description of the differential}\label{sec:differential}
We are now ready to state the main result of this paper. Recall that $\Lambda$ now represents the dipped version of the link of Legendrian spheres.

Let 
\begin{equation*}
\Al_D = \Al(J^1(\R^{n-1}),\Lambda \cap D_a), \qquad\Al_H = \Al(J^1(\Lambda_\st), \Lambda \cap \Ha^k_+)
\end{equation*}
where we view the whole dipping region as contained in $\Ha^k$, and the part where the minimum in the $\rho$-direction is obtained as also contained in $D_a$. Moreover, let $\Al_S$ be the DGA generated by the chords $b_1[p], \dotsc, b_m[p]$, $p = m_2, s_2$, and with differential given by the restriction of the differential on  $\Al(D_A^H, \Lambda)$.   

\begin{proposition}\label{prp:prp1}
The DGAs $\Al_D, \Al_H$ are sub-DGAs of $\Al(D_A^H, \Lambda)$ and $\Al_S$ is a sub-DGA of both $\Al_D$ and $\Al_H$. Moreover, $\Al(D_A^H, \Lambda)$ is the pushout of the inclusion maps of these sub-DGAs, that is, it is the pushout of the following diagram:
\begin{equation}\label{eq:pushout1}
\begin{tikzcd} 
	\Al_S \arrow {d}{i} \arrow {r}{i} & \Al_H \arrow {d}{i} \\
	\Al_D \arrow {r}{i} & \Al(D_A^H, \Lambda).
\end{tikzcd}
\end{equation}

\end{proposition}

Since Lemma \ref{lma:quasi} says that $\Al(V,\Lambda)$ is quasi-isomorphic to  $\Al(D_A^H, \Lambda)$, we obtain our main theorem.

\begin{corollary}[Theorem \ref{thm:main}]
 The DGA $\Al(V,\Lambda)$ is quasi-isomorphic to the pushout of the following diagram 
 \begin{equation*}
\begin{tikzcd} 
	\Al_S \arrow {d}{i} \arrow {r}{i} & \Al_H  \\
	\Al_D, & 
\end{tikzcd}
\end{equation*} 
 where $i$ is the inclusion and where $\Al_S, \Al_D, \Al_H$ can be computed from Legendrians in one-jet spaces.
\end{corollary}

\begin{proof}[Proof of Proposition \ref{prp:prp1}]

Similar to \cite{en} the differential  of $ \Al(D_A^H, \Lambda)$ can be given by a count of pseudo-holomorphic disks in $\R \times D_a^H$ with boundary on $\R \times \Lambda$. After having introduced the dipping region this count reduces to the following.

\begin{description}
 \item[Diagram disks] Disks that have positive puncture at the diagram chords $a_1, \dotsc, a_{\tilde m}$. These disks cannot leave $\R \times D_a$ because of the dipping (they cannot pass through $S^{k-1} \times \Lambda_\sub \times \{p_2\}$), which is easily seen by switching to Morse flow trees instead of pseudo-holomorphic disks. Hence they can only have negative punctures at other diagram chords  or at the chords $b_i[p]$, $i=1, \dotsc, m$, $p= s_2,m_2$ (reall that these chords are assumed to be contained in $D_a$).

\item[Handle disks]
These are disks with positive punctures at the handle chords. In Section \ref{sec:setup} we prove that we only have to consider  disks with negative punctures at handle chords.

\item[Dipping disks]
These are pseudo-holomorphic disks of $\Lambda$ having positive puncture at some  $b_i[m_1]$ or $b_i[s_1]$, $i=1, \dotsc, m$.

The only way for these disks to leave the dipping region is to enter the handle. If the parameters for the dipping function $f$ is chosen sufficiently small it follows by action reasons that these disks cannot leave a 1--jet neighborhood of $\Lambda_\st \subset \Ha^k_+$ and hence we can use the techniques from Section \ref{sec:1jet} to find the pseudo-holomorphic disks. Note that we may identify the whole dipping area with a subset of a 1--jet neighborhood of $\Lambda_\st \subset \Ha^k_+$ if the dipping parameters are small enough. 

In particular, it follows that these disks can only have negative punctures at dipping chords or at handle chords.

\item[Minimum disks]  These are disks with positive punctures at some $b_i[p]$, $i=1, \dotsc, m$, $p= s_2,m_2$. By passing to Morse flow trees we see that these disks cannot leave $S^{k-1} \times \Lambda_\sub \times \{p_2\}$, and in particular these disks can only have negative punctures at the chords  $b_i[p]$, $i=1, \dotsc, m$, $p= s_2,m_2$.
\end{description}

From this it follows that $\Al_D$ and $\Al_H$ are sub-DGAs of $\Al(D_A^H, \Lambda)$ and that $\Al_S$ is a sub-DGA of both  $\Al_D$ and $\Al_H$.

To prove that \eqref{eq:pushout1} describes $\Al(D_A^H, \Lambda)$ as a pushout, we introduce the notation $|\Al|$ for the set of generators of the DGA $\Al$. 
Now, since 
\begin{align*}
|\Al(D_A^H, \Lambda)| = |\Al_H| \cup |\Al_D|, \\  
|\Al_S|= |\Al_H| \cap |\Al_D|
\end{align*}
it follows that $\Al(D_A^H, \Lambda)$ is the pushout of \eqref{eq:pushout1} as an algebra. Since the disks occuring in the differential of $\Al(D_A^H, \Lambda)$ are of the form listed above, it follows that if $a \in |\Al(D_A^H, \Lambda)|$, then
\begin{align*}
\partial(a) &= \partial|_{\Al_D}(a), \qquad \text{if } a \in |\Al_D|\\
\partial(a) &= \partial|_{\Al_H}(a), \qquad \text{if } a \in |\Al_H|.
\end{align*}
And since $\partial|_{\Al_D} = \partial|_{\Al_H}$ on $\Al_S$ it follows that $\Al(D_A^H, \Lambda)$ is the pushout of \eqref{eq:pushout1} also as a DGA. 
\end{proof}

\begin{remark}
Note that we have four copies of $\Al(J^1(\R^{n-k-1}), \Lambda_\sub)$ in the dipping region, and one in the handle. The copies sitting at $h$ and $m_2$ will be subalgebras, but the other ones will not (except in the case when $k=1$, where we also get a subalgebra at $s_2$). However, it is possible to describe how these copies interact with each other using the techniques of Morse-Bott cascades from \cite{tobkal}. This will be done in a forthcoming paper, and with this we also get a tool to treat the $k=(n-1)$--case.
\end{remark}

It follows that the Legendrian contact homology of $\Lambda$ in $V$ can be computed using 1--jet space techniques. It also follows that the coefficients reduce form  $\Z_2[H^2(X,\Lambda)]$ to $\Z_2$ since no disk passes through a handle.

\section{The sub-DGA in the handle}\label{sec:setup}
In this section we prove that we get a sub-DGA of $\Al(D_A^H, \Lambda)$ generated by the Reeb chords in  the cocore of the handle. To do this, we will modify the model of the handles from Section \ref{sec:geometry} slightly, to simplify the Reeb dynamics.

\subsection{Geometry of a $2n$-dimensional symplectic handle of index k}

Let \linebreak $a_1, \dotsc, a_{n-k} \in \R$ be some positive constants that are linearly independent over $\Q$ and define
\begin{equation*}
 H_\delta^k =\{(x,y,p,q)\in \R^{2n-2k} \times \R^{2k}: -\delta^2 \leq \sum_{j=1}^{n-k}\frac{a_j}{2}(x_j^2 + y_j^2) +  \sum_{j=1}^k 2p_j^2 - q_j^2\leq \delta^2\}. 
\end{equation*}

The handle still has Liouville vector field
\begin{equation}
 Z = \sum_{j=1}^{n-k}\frac{1}{2}(x_j \partial_{x_j} + y_j\partial_{y_j}) + \sum_{j=1}^k  2p_j\partial_{p_j} - q_j\partial_{q_j}
\end{equation} 
which is transverse to the boundary
\begin{equation*}
\Ha_{\pm}^k= \{(x,y,p,q)\in \R^{2n-2k} \times \R^{2k}:  \sum_{j=1}^{n-k}\frac{a_j}{2}(x_j^2 + y_j^2) +  \sum_{j=1}^k 2p_j^2 - q_j^2 = \pm \delta^2\},
\end{equation*}
and points out of $H_\delta^k$ along $\Ha_+^k$ and into $H_\delta^k$ along $\Ha_{-}^k$. 
Note that, topologically we still have $\Ha_{-}^k \simeq \R^{2n-k} \times S^{k-1}$, $\Ha_+^k \simeq S^{2n-k-1} \times \R^{k}$. 

The Liouville vector field induces contact forms $\alpha_{\pm\delta}$ on $\Ha_{\pm}^k$:
\begin{equation}
 \alpha_{\pm \delta} = \omega_{st}(Z,\bullet)|_{\Ha_{\pm}^k} = \sum_{j=1}^{n-k}\frac{1}{2}(x_j d{y_j} - y_jd{x_j}) +  \sum_{j=1}^k 2p_jd{q_j} + q_jd{p_j},
\end{equation}
with Reeb vector field given by $N\tilde R$, where
\begin{equation}
 \tilde R =  \sum_{j=1}^{n-k}\frac{a_j}{2}(x_j \partial_{y_j} - y_j\partial_{x_j}) +  \sum_{j=1}^{k}2p_j\partial_{q_j} + q_j\partial_{p_j}|_{\Ha_{\pm}^k}
\end{equation} 
and
\begin{equation*}
 N = \left(\sum_{j=1}^{n-k}\frac{a_j}{4}(x_j^2+y_j^2) + \sum_{j=1}^{k}4p_j^2 + q_j^2\right)^{-1}.
\end{equation*}

It follows that the differential equation for the Reeb flow is given by 
\begin{align*}
 \dot x_j & = - \frac{Na_j}{2}y_j &  \dot y_j & = \frac{Na_j}{2}x_j  & j =1,\dotsc, n-k,\\
 \dot p_j & = Nq_j &   \dot q_j & = 2Np_j & j = 1,\dotsc,k. 
\end{align*}

Hence we get that the time $t$ Reeb flow  $\Phi^t_R=(x(t),y(t),p(t),q(t))$ is given by 
\begin{align}\label{eq:R1}
 x_j(t) &= x_j(0) \cos\left(\frac{Na_j}{2}t\right) - y_j(0) \sin\left(\frac{Na_j}{2}t\right), \qquad j =1, \dotsc, n-k, \\ \label{eq:R2}
 y_j(t) &=  x_j(0) \sin\left(\frac{Na_j}{2}t\right) + y_j(0) \cos\left(\frac{Na_j}{2}t\right), \qquad j =1, \dotsc, n-k,\\ \label{eq:R3}
 p_j(t) &= p_j(0) \cosh(N\sqrt{2}t) + \frac{1}{\sqrt{2}}q_j(0) \sinh(N\sqrt{2}t), \qquad j =1, \dotsc, k, \\ \label{eq:R4}
 q_j(t) &= \sqrt 2 p_j(0) \sinh(N\sqrt{2}t) +q_j(0) \cosh(N\sqrt{2}t), \qquad j =1, \dotsc, k. 
\end{align}

Let us now consider some special cases.

\subsection{Index 1 handles}
In this case the attaching region is given by two disjoint balls of dimension $2n-1$. 
We describe models for the attachment of the handle $\Ha = \Ha^1$ along these balls.

Let $\bar a = (a_1,\dotsc, a_{n-1})$ and let 
\begin{equation*}
B^{2n-1}_{\rho,\bar a} = \{(u_1,v_1, \dotsc,, u_{n-1}, v_{n-1}, z) \in \R^{2n-1}; \sum_{i=1}^{n-1}a_i(u_i^2 + v_i^2) + z^2 \leq \rho \} 
\end{equation*}
equipped with the contact structure $\alpha_b = dz + \frac{1}{2}(u dv - v du)$. 
We will often omit the dimension and only write $B_{\rho,\bar a}$.
Identify this with a ball in $(\R^{2n-1}, dz - ydx)$, centered at $c = (x,y,z_0)$ via the contact embedding $F_c \colon B_{\rho,\bar a} \to \R^{2n-1}$,
\begin{equation*}
 F_c(u,v,z) = (x+u,y+v,z+z_0+yu + \frac{1}{2}uv).
\end{equation*} 

For suitable $\rho$ and $c$ these balls will be our attaching locus, as follows. 
Fix two points $c_{\pm} \in \R^{2n-1}$ of distance $d >> \rho$ and a $\delta << \rho$, let $A^\pm_\rho(\delta) \subset \Ha_-$ be defined by
\begin{align*}
 A^+_\rho(\delta) &= \{(x,y,p,q) \in \Ha_-; p = 0, q> 0, \sum_{i=1}^{n-1} a_i(x_i^2 + y_i^2)\leq \rho^2\}, \\
 A^-_\rho(\delta) &= \{(x,y,p,q) \in \Ha_-; p = 0, q< 0, \sum_{i=1}^{n-1} a_i(x_i^2 + y_i^2)\leq \rho^2\}. 
\end{align*}
Define the map $G \colon A^\pm_\rho(\delta) \to B_{\rho,\bar a}$,
\begin{equation*}
 G(x,y,0,q) = ( x_1, y_1,\dotsc,  x_{n-1}, y_{n-1},0).
\end{equation*}
Then $G^*\alpha_b = \alpha_{-\delta}|_{A^\pm_\rho(\delta)}$ and since the Reeb vector field is transverse to $A^\pm_\rho(\delta)$ and $G(A^\pm_\rho(\delta))$, respectively, its flow can be used to extend $G$ to a contactomorphism from a neighborhood of $A^\pm_\rho(\delta)$ to $B_{\rho,\bar a}$. 
Denote this neighborhood by $B^{\pm}_{\rho,\bar a}(\delta) \subset \Ha_-$, and note that this neighborhood is identified with a neighborhood of $c_\pm$ via the composition $F_{c_\pm} \circ G$.

The next step is to identify $B^{\pm}_{\rho,\bar a}(\delta) - B^{\pm}_{\frac{\rho}{2}, \bar a}(\delta) \subset \Ha_-$ with a region in $\Ha_+$. 
To do that we use the Liouville flow in the handle, given by 
\begin{align*}
 &\phi_Z^t((x(0),y(0),p(0),q(0)) = \\
 &(e^{\frac{1}{2}t}x_1(0),e^{\frac{1}{2}t}y_1(0), \dotsc , e^{\frac{1}{2}t}x_{n-1}(0),e^{\frac{1}{2}t}y_{n-1}(0),e^{{2}t}p(0),e^{-t}q(0)).
\end{align*}

Recall the standard Legendrian cylinder defined in Section \ref{sec:assumptions}. Perturb it to be given by 
\begin{align*}
 \Lambda_\st &= \{(x,y,p,q) \in \Ha_+; y = p = 0 \} = \{(x,y,p,q) \in \Ha_+; \sum_{i=1}^{n-1}a_ix_i^2 =2(\delta^2 + q^2) \} \\
 & = \left\{\left(\sqrt{\frac{2(\delta^2 + q^2)}{a_1}}\hat x_1,0,\dotsc,\sqrt{\frac{2(\delta^2 + q^2)}{a_{n-1}}}\hat x_{n-1},0, 0,q\right)\right\}
\end{align*}

where $\hat x_1,\dotsc,\hat x_{n-1}$ are coordinates for $S^{n-2}\subset \R^{n-1}$. 

Let $T(x,y,p,q)$ be the time of the Liouville flow from $\Ha_-$ to $(x,y,p,q)\in \Ha_+$.

\begin{lemma}\label{lma:T}
 If $a = \left(\sqrt{\frac{2(\delta^2 + q^2)}{a_1}}\hat x_1,0,\dotsc,\sqrt{\frac{2(\delta^2 + q^2)}{a_{n-1}}}\hat x_{n-1},0, 0,q\right) \in \Lambda_\st$ then $T(a) = T(q)$ and $T(q)$ decreases when $q$ increases ($e^{-T(q)}$ increases with $q$).
\end{lemma}

\begin{proof}
We should find points $x_i,q_0 \in \Ha_-$, $i=1,\dotsc, n-1$, so that 
\begin{equation}
x_i = e^{-\frac{1}{2}T(a)}\sqrt{\frac{2(\delta^2 + q^2)}{a_i}} \hat x_i, \quad i = 1, \dotsc, n-1, \qquad q_0 = e^{T(a)}q.
\end{equation}
and such that 
\begin{equation}\label{eq:sumt}
\sum_{j=1}^{n-1} \frac{a_j}{2} x_j^2 - q_0^2 = -\delta^2.
\end{equation}
That means that $T(a)$ should satisfy  
\begin{equation}
\sum_{j=1}^{n-1} \frac{a_j}{2} e^{-T(a)}\frac{2(\delta^2 + q^2)}{a_j} \hat x_j^2  = e^{2T(a)}q^2-\delta^2,
\end{equation}
which simplifies to
\begin{equation}\label{eq:sum2}
 e^{-T(a)}(\delta^2 + q^2) = e^{2T(a)}q^2-\delta^2.
 \end{equation} 
 From this we see that $T(a) = T(q)$. Write $u = e^{T(a)}$ to simplify notation. Then \eqref{eq:sum2} can be written as 
 \begin{equation}
 \delta^2 + q^2  = u^3q^2- u\delta^2.
 \end{equation}
 When $q^2 = \delta^2$ this equation has solution close to $u =3/2$. Moreover, assuming that $u > 1$ we might rewrite this as 
 \begin{equation}
 \frac{q^2}{\delta^2} = \frac{u+1}{u^3-1},
 \end{equation}
 and since the function $\frac{u+1}{u^3-1}$ is strictly increasing to $\infty$ as $u$ decreases from $2$ to $1$ , the second statement follows.
\end{proof}

This means that the image of $\Lambda_\st$ in $\Ha_-$ under the negative Liouville flow is given by 
\begin{equation*}
 \left(e^{-\frac{1}{2}T(q)}\sqrt{\frac{2(\delta^2 + q^2)}{a_1}}\hat x_1,0,\dotsc,e^{-\frac{1}{2}T(q)}\sqrt{\frac{2(\delta^2 + q^2)}{a_{n-1}}}\hat x_{n-1},0, 0,e^{T(q)}q\right).
\end{equation*}
To view this in $B_{\rho,\bar a}$ let 
\begin{equation*}
E^{2n-2}(\rho,\bar a) = \{(u_1,v_1, \dotsc,, u_{n-1}, v_{n-1}, z) \in \R^{2n-1}; \sum_{i=1}^{n-1}a_i(u_i^2 + v_i^2) + z^2 = \rho \}.  
\end{equation*}
Then 
\begin{equation*}
 \partial B_{\rho,\bar a}^{2n-1} = E^{2n-2}(\rho,\bar a) \supset E^{2n-3}(\rho,\bar a) = \{\sum_{i=1}^{n-1} a_i(u_i^2 + v_i^2) = \rho\},
\end{equation*}
which is a contact submanifold of $B_{\rho,\bar a}^{2n-1}$. Moreover, after scaling if necessarily, we get that 
\begin{equation*}
 \Lambda_\st \cap E^{2n-3}(\rho,\bar a) = \{\sum_{i=1}^{n-1} a_iu_i^2 = \rho, v= 0\}  = E^{n-2}(\rho,\bar a)
\end{equation*}
which is a Legendrian submanifold of $E^{2n-3}(\rho,\bar a) $, denote it $\Lambda_{\st, \sub}$. 
Let $U_\nu \subset E^{2n-3}(\rho,\bar a) $ be a $1$-jet neighborhood of $\Lambda_{\st,\sub}$.

Assume that 
\begin{equation*}
\Lambda \cap \partial B_{\rho,\bar a} = \Lambda \cap E^{2n-3}(\rho,\bar a) = \Lambda \cap U_\nu  
\end{equation*}
 which then is a Legendrian in $E^{2n-3}(\rho,\bar a)$, which we denote by $\Lambda_\sub$. Assume that it does not intersect any coordinate subspaces $\{(u_i,v_i) = 0\}$ for $i =1,\dotsc,n-1$.  Further assume that $\Lambda \cap B_{\rho,\bar a}$ is a cone on $\Lambda_\sub$, meaning that if 
 \begin{equation*}
  \Lambda_\sub = \{(\tilde u, \tilde v)\} \subset E^{2n-3}(\rho,\bar a) \subset \partial B_{\rho,\bar a},
 \end{equation*}
then there is some  positive parameter $r$ which is strictly
increasing with the radius of $B^{2n-1}_{\rho,\bar a}$ so that 
\begin{equation*}
 \Lambda \cap (B_{\rho,\bar a}^{2n-1} \setminus B_{\frac{1}{2}\rho, \bar a}^{2n-1}) = \{(r \tilde u, r \tilde v,0)\}.
\end{equation*}

Hence
\begin{equation*}
 G^{-1}(\Lambda_{\sub}) = \{(\tilde u, \tilde v, 0, 0) \} \subset \Ha_-
\end{equation*}
and similar to the proof of Lemma \ref{lma:T} we see that we might assume that $\Lambda$ is a cone on $\Lambda_{\sub}$ in the handle, that is
\begin{equation*}
 \Lambda \cap \Ha_+ = \{(\tilde r\tilde u, \tilde r\tilde v, 0, q) \}, 
\end{equation*}
where $\tilde r \colon D^1 \to \R_{\geq 0}$ is a Morse function with exactly one critical point, located at $q = 0$ and of index 0.  
.

Let $E_\delta(a_1, \dotsc, a_{m}) = \{\sum_{i=1}^{m} a_i |z_i|^2 =2\delta^2\} \subset \C^m$. 

\begin{lemma}\label{lma:subalg1}
 The Reeb chords of $\Lambda \cap \Ha_+  \subset \Ha_+$ can be identified with the Reeb chords of $G^{-1}(\Lambda_{\sub}) \subset E_\delta(a_1, \dotsc, a_{n-1})$, with degree shift given as in Lemma \ref{lma:gradingb}.
\end{lemma}

\begin{proof}
 Since $p = 0$ along $\Lambda \cap \Ha_+ $ we must have $p=q=0$ along Reeb chords of $\Lambda \cap \Ha_+ $, since the Reeb flow is hyperbolic in the $(p,q)$-plane. 
 Thus the Reeb dynamics restrict to the one on $E_\delta(a_1,\dotsc,a_{n-1})$. 
\end{proof}

\begin{lemma}\label{lma:inhandle}
 If we have a $J$-holomorphic disk of $\Lambda$ with a positive puncture in $\Ha$ then it cannot have a negative puncture outside $\Ha$. 
\end{lemma}
\begin{proof}
 This follows similarly as in the proof of [\cite{en}, Lemma 5.6]. Note that the $m$:th iterate of a Reeb orbit in the handle has grading of order $m$ and action of order $\delta m$, which is what is needed for the proof to work out.  
\end{proof}

\begin{lemma}\label{lma:index1sub} 
Assume that $n>2$. Then the DGA $\Al(\Ha_+,\Lambda\cap \Ha_+)$ is  a sub-DGA of $\Al(D_A^H,\Lambda)$, and this sub-DGA is quasi-isomorphic to $\Al(J^1(\R^{n-2}),\Lambda_{\sub})$ with a degree shift given as in Lemma \ref{lma:gradingb}.
\end{lemma}
\begin{proof}
 The first statement follows from Lemma \ref{lma:inhandle}, and since $\Lambda$ is conic in $\Ha^+$ with a minimum at $q=0$ it follows that $\Al(\Ha_+,\Lambda\cap \Ha_+)$ is quasi-isomorphic to \linebreak 
 $\Al(E_\delta(a_1, \dotsc, a_{n-1}) ,G^{-1}(\Lambda_{\sub}))$ with a degree shift given as in Lemma \ref{lma:gradingb}, which in turn is quasi-isomorphic to  $\Al(J^1(\R^{n-2}),\Lambda_{\sub})$  by Lemma \ref{lma:quasi}. 

\end{proof}

\subsection{Index $k$ handles, $k \in \{2, \dotsc, n-2\}$}

This is similar to the case of index 1 handles, so we just describe the main differences.

To describe the attaching map, recall the contact identification 
\begin{equation*}
\psi \colon N(\Upsilon) \to  T^*_{\rho_1}(S^{k-1})\times D_{\rho_2}^{2n-2k} \times I_{\rho_3}
\end{equation*}
from Section \ref{sec:assumptions}. 

Let $\rho = (\tilde \rho_1, \tilde \rho_2)$ and let
\begin{equation}
A_\rho(\delta) = \{(x,y,p,q) \in \Ha_-; p \cdot q = 0, \sum_{i=1}^{n-k}a_i(x_i^2 + y_i^2) < \tilde \rho_2, |p|^2 < \tilde \rho_1\}.
\end{equation}
Choose  $\tilde \rho_1, \tilde \rho_2$ so small so that $G \colon A_\rho(\delta) \to  T^*_{\rho_1}(S^{k-1})\times D_{\rho_2}^{2n-2k} \times I_{\rho_3}$, 
\begin{align*}
G(x,y,p,q)  =& (q_1,\dotsc, q_{k}, (p\cdot q)q_1-p_1,\dotsc, (p\cdot q)q_k-p_k, x_1,\dotsc, x_{n-k},\\
&y_1,\dotsc, y_{n-k}, p\cdot q+\sum_{i=1}^{n-k} \frac{x_iy_i}{2})
\end{align*}
is defined.
Then $G^*\alpha_N = \alpha_{-\delta}|_{A_\rho(\delta)}$ and since the Reeb vector fields are transverse to $A_\rho(\delta)$ and $G(A_\rho(\delta))$, respectively, we can extend $G$ to a contactomorphism from a neighborhood of $A_\rho(\delta)$ in $\Ha_-$ to a neighborhood of $G(A_\rho(\delta))$.

As in the case of index 1 handles we might assume that $\Lambda$ is of the form
\begin{equation}
\Lambda \cap \Ha_+ = (\tilde r(q) G^{-1}(\Lambda_\sub),0,q) \simeq \Lambda_\sub \times D_k \subset \Ha_+ 
\end{equation}
where ${\Lambda}_\sub \subset E^{2n-2k-1}(\tilde \rho_2,\bar a)$ is a Legendrian submanifold and $\tilde r \colon D^k \to \R_{\geq 0}$ is a Morse function with exactly one critical point, located at $q = 0$ and of index 0. Moreover, we assume that  $\Lambda \cap \Ha_+$ is contained in a small 1--jet neighborhood of the perturbed standard Legendrian cylinder in $\Ha$
\begin{align*}
 \Lambda_\st &= \{(x,y,p,q) \in \Ha_+^k; y = p = 0 \} = \{(x,y,p,q) \in \Ha_+^k;\sum_{i=1}^{n-k}a_ix_i^2 = \delta^2 + \sum_{i=1}^k q_i^2\} \\
 &\simeq S^{n-k-1} \times D^k. 
\end{align*}

\begin{lemma}\label{lma:subalgk}
The Reeb chords of $\Lambda \cap \Ha_+$ are contained in the slice $p=q=0$, and $\Al(\Ha_+,\Lambda \cap \Ha_+)$ is a sub-DGA of $\Al(V,\Lambda)$ isomorphic to $\Al(J^1(\R^{n-k-1}),\Lambda_\sub)$ with a degree shift as in Lemma \ref{lma:gradingb}.
\end{lemma}
\begin{proof}
This is similar to the case of index 1 handles.
\end{proof}

\section{The singular homology of the free loop space of $\Cp^2$}\label{sec:examples}
In this section we give a Weinstein handle decomposition of $T^*\Cp^2$ and compute the Chekanov-Eliashberg DGA of the Legendrian attaching sphere. By the results of \cite{cotangentloop,cotangentloopviterbo, cotangentloopweber} and \cite{surg1} this gives a description of the singular homology of the free loop space of $\Cp^2$.

\subsection{Weinstein handle decomposition of $T^*\Cp^2$}
Recall that $\Cp^2$ is obtained from $B^4$ by attaching a 2-handle along a knot $\tilde \Upsilon \subset S^3$ with framing 1, and then attaching a 4-handle along the boundary of $B^4 \cup_{\tilde \Upsilon} \Ha^2$. 
Thus, to give a Weinstein handle decomposition of $T^*\Cp^2$ we should attach one subcritical handle to $\partial B^8$ along an isotropic $S^1$ with the correct framing and then attach the critical handle along a Legendrian  $S^3$ which goes through the subcritical handle along the standard Legendrian cylinder.

Another way of seeing this is to consider the Legendrian 
\begin{equation*}
 \tilde \Lambda = \{x_1^2 + x_2^2 + x_3^2 + x_4^2 = 1\} \subset S^7 = \{z = x + iy \in \C^4; |z| = 1\}
\end{equation*}
with the isotropic attaching sphere $\Upsilon$ of the subcritical handle as being a subset of $\tilde \Lambda$, and then let $\Lambda$ be the Legendrian submanifold we get by replacing a neighborhood of $\Upsilon$ with the standard Legendrian cylinder in $\Ha^2$.

If we view $\tilde \Lambda$ as a Legendrian submanifold of $J^1(\R^3) \simeq \partial B^8 \setminus \{\text{pt}\}$, then it can be given by the 1--jet lift of the locally defined functions $f_\pm \colon \R^3 \to \R$, $f_\pm(u_1, u_2, u_3)  = \tilde f_\pm(|u|)$ with $\tilde f_\pm$ as in Figure \ref{fig:fpm}. (Recall that the 1--jet lift of $f\colon \R^n \to \R$ is given by $\{(u,df(u),f(u)); u \in \R^n\} \subset J^1(\R^n)$.)


  \begin{figure}[ht]
 \includegraphics[height=10cm, width=13cm]{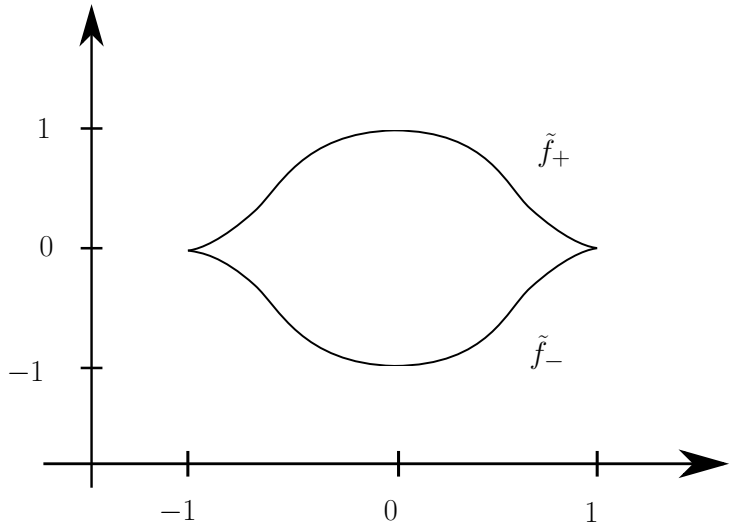}
 \vspace{-1cm}
\caption{The locally defined functions $\tilde f_\pm$.}
 \label{fig:fpm}
 \end{figure}

We see that $\tilde \Lambda$ has exactly one Reeb chord $a$ of grading 3, and this Reeb chord projects to the origin under the base projection $\Pi: J^1(\R^3) \to \R^3$. We also see that there is an $S^2$--family of Morse flow trees going from $\Pi(a)$ to the base projection of the cusp edge, which is nothing but the unit sphere $S^2 \subset \R^3$. In addition, there are no more flow trees of $\tilde \Lambda$. 

Identifying $\tilde \Lambda \setminus \{(0,0,1)\}$ with $\R^3$ via stereographic projection (mapping $(0,0,-1)$ to the origin) we get that the lifts of the flow trees of $\tilde \Lambda$ go radially from $\infty$ to the origin. Assume that the unit sphere $S^2 \subset \R^3$ represents the cusp edge singularity of $\tilde \Lambda$. 

Now we describe our choice of attaching sphere $\Upsilon$ for the subcritical handle in this picture. 
To that end, pick a point $p \in S^2 \subset \R^3 \subset \tilde \Lambda$. 
Then a neighborhood of $p$ in $\R^3$ can be identified with $J^1(I) \simeq T^*I \times \R$ where $I\subset S^2$ is some interval and $T^*I \simeq D$ is a disk contained in $S^2$, and the $\R$--direction corresponds to the radial direction in $\R^3$. 
Let $\Upsilon$ be given by the standard Legendrian unknot as in Figure \ref{fig:unknot}.

  \begin{figure}[ht]
 \includegraphics[height=6cm, width=7.5cm]{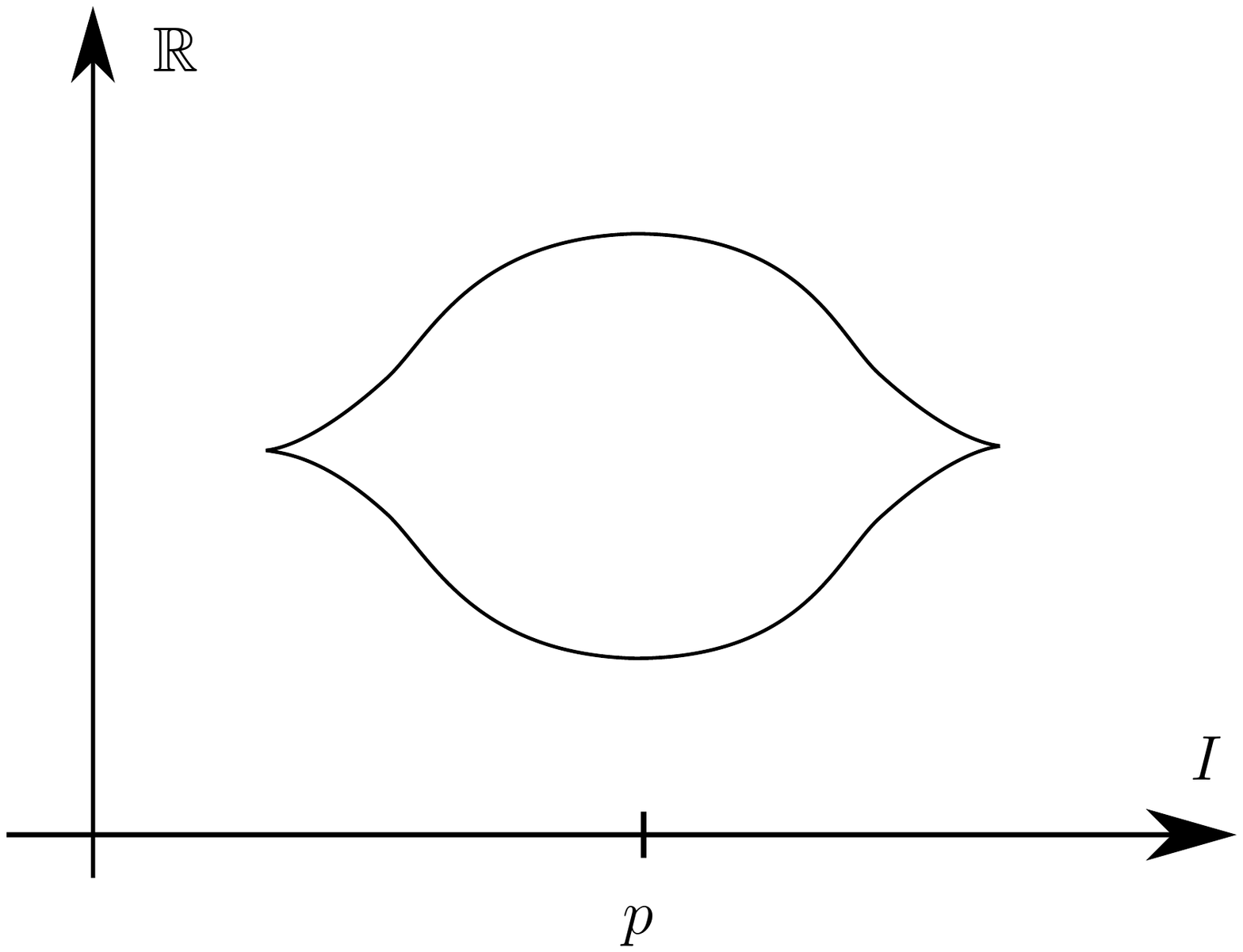}
\includegraphics[height=6cm, width=7.5cm]{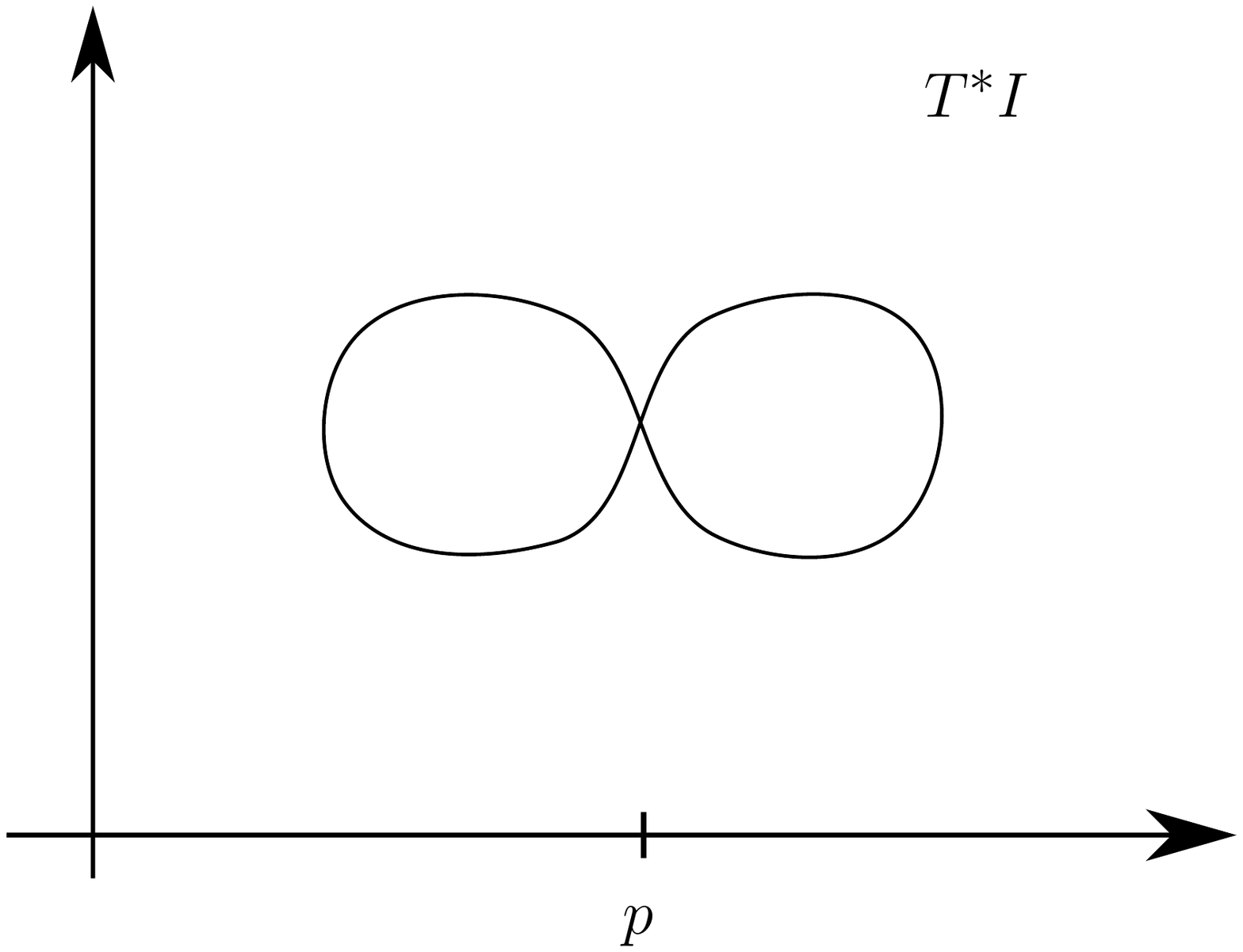}
 \vspace{-2cm}
\caption{The front and Lagrangian projection of the standard Legendrian unknot in $J^1(I)$.}
 \label{fig:unknot}
 \end{figure}

It follows that the lifted flow trees of $\tilde \Lambda$ intersect $\Upsilon$ either 0,1 or 2 times, in 2--dimensional, 1--dimensional and 0--dimensional families, respectively. 
Moreover, there is exactly one tree that intersects $\Upsilon$ in 2 points, namely the tree which goes through the point $p$. This tree gives rise to a rigid flow tree $\Gamma$ of $\Lambda$, and this will be the only rigid flow tree. Since $\Lambda$ coincides with $\Lambda_\st$ in the subcritical handle we have that $\Lambda_\sub$ is given by the standard Legendrian unknot in Figure \ref{fig:unknot}, which has exactly one Reeb chord $b$, and $|b|_{\sub} = |b| = 1$. The tree $\Gamma$ has positive puncture at $a$, one $Y_0$-vertex over the point $p$ and 2 negative punctures at $b[h]$. See Figure \ref{fig:tree}.


  \begin{figure}[ht]
 \includegraphics[height=6cm, width=7.5cm]{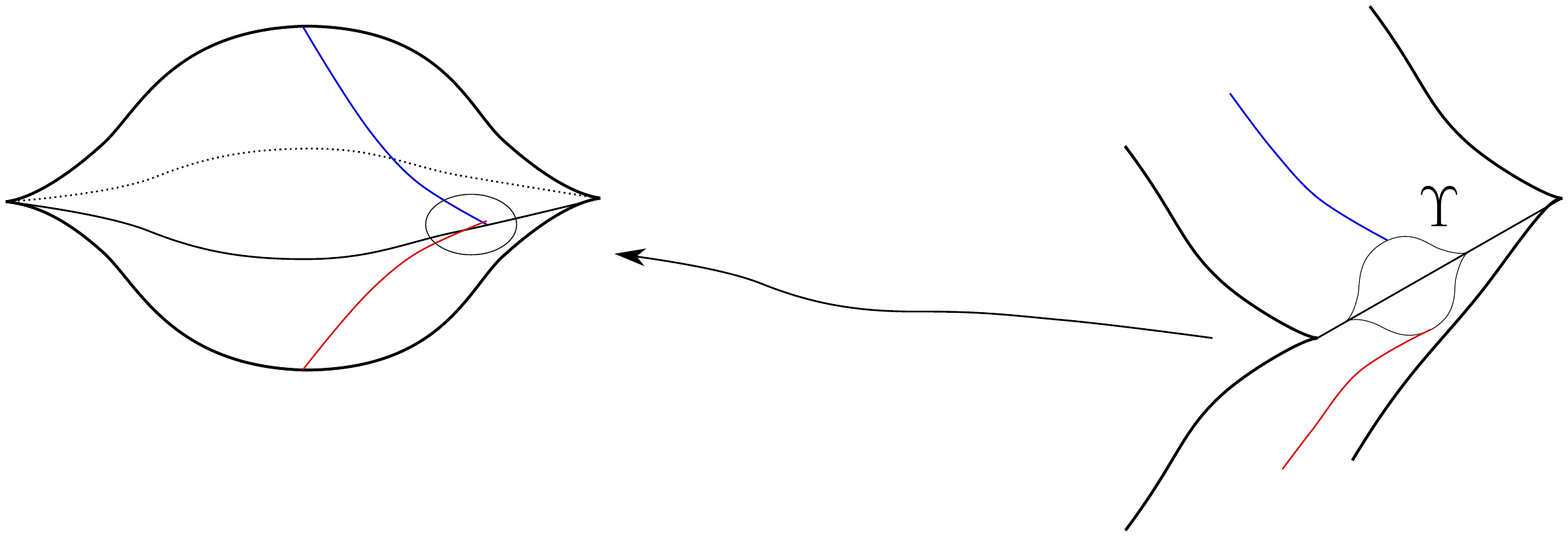}
 \includegraphics[height=6cm, width=7.5cm]{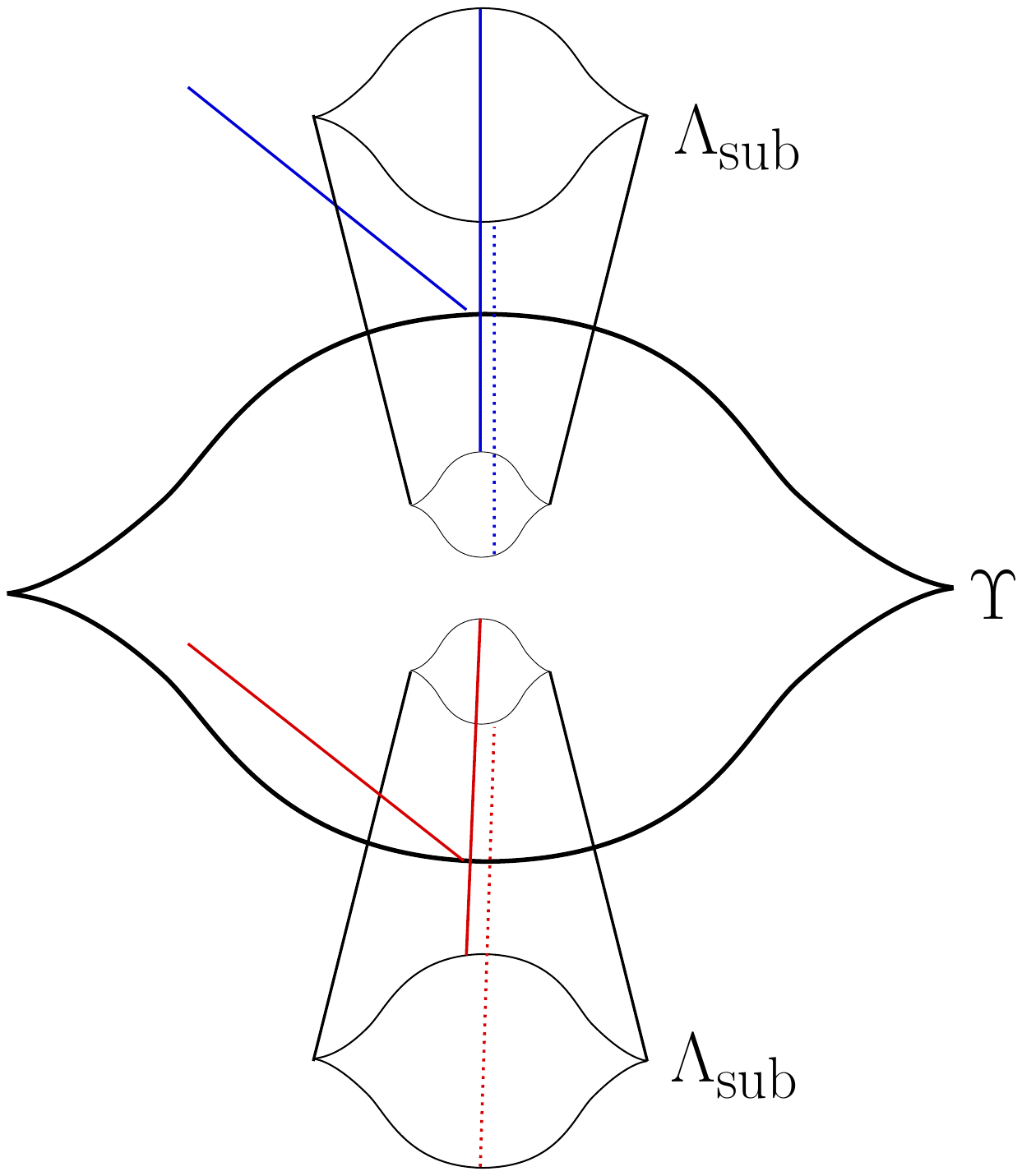}
 \vspace{-1cm}
\caption{The lift of $\Gamma$ in blue, red in $\tilde \Lambda$ in the left hand picture. In the right hand picture we see how the flow line splits at a $Y_0$-vertex and then travels into the handle along the standard Legendrian.}
 \label{fig:tree}
 \end{figure}

 It follows that $\Al(\Lambda)$ is generated by the chords $a,b$ with $|a| =3$, $|b| = 1$ and where
\begin{align}
 \partial a &= b^2, \\
 \partial b &= 0.
\end{align}
By [\cite{surg1}, Corollary 5.7] we get that the Hochschild homology of this DGA coincides with the symplectic homology of $T^*\Cp^2$. This, in turn, gives the singular homology of the free loop space of $\Cp^2$ by the results in \cite{cotangentloop,cotangentloopviterbo, cotangentloopweber}.

 \bibliographystyle{halpha}
 \bibliography{c_main_nn} 

\newcommand{\etalchar}[1]{$^{#1}$}
\def\cprime{$'$}
\begin{thebibliography}{CDGG17}

\bibitem[ACG{\etalchar{+}}20]{lauraemmy}
Bahar {Acu}, Orsola {Capovilla-Searle}, Agn{\`e}s {Gadbled}, Aleksandra
  {Marinkovi{\'c}}, Emmy {Murphy}, Laura {Starkston}, and Angela {Wu}.
\newblock {An introduction to Weinstein handlebodies for complements of
  smoothed toric divisors}.
\newblock {\em arXiv e-prints}, page arXiv:2002.07983, February 2020,
  2002.07983.

\bibitem[AS06]{cotangentloop}
Alberto Abbondandolo and Matthias Schwarz.
\newblock On the {F}loer homology of cotangent bundles.
\newblock {\em Comm. Pure Appl. Math.}, 59(2):254--316, 2006.

\bibitem[BEE12]{surg1}
Fr\'ed\'eric Bourgeois, Tobias Ekholm, and Yasha Eliashberg.
\newblock Effect of {L}egendrian surgery.
\newblock {\em Geom. Topol.}, 16(1):301--389, 2012.
\newblock With an appendix by Sheel Ganatra and Maksim Maydanskiy.

\bibitem[CDGG17]{geomgen}
Baptiste {Chantraine}, Georgios {Dimitroglou Rizell}, Paolo {Ghiggini}, and
  Roman {Golovko}.
\newblock {Geometric generation of the wrapped Fukaya category of Weinstein
  manifolds and sectors}.
\newblock {\em arXiv e-prints}, page arXiv:1712.09126, Dec 2017, 1712.09126.

\bibitem[CM19]{rm}
Roger Casals and Emmy Murphy.
\newblock Legendrian fronts for affine varieties.
\newblock {\em Duke Math. J.}, 168(2):225--323, 2019.

\bibitem[DR16]{Georgios}
Georgios Dimitroglou~Rizell.
\newblock Lifting pseudo-holomorphic polygons to the symplectisation of
  {$P\times\Bbb{R}$} and applications.
\newblock {\em Quantum Topol.}, 7(1):29--105, 2016.

\bibitem[EES05]{legsub}
Tobias Ekholm, John Etnyre, and Michael Sullivan.
\newblock The contact homology of {L}egendrian submanifolds in {${\mathbb
  {R}}^{2n+1}$}.
\newblock {\em J. Differential Geom.}, 71(2):177--305, 2005.

\bibitem[EES07]{PR}
Tobias Ekholm, John Etnyre, and Michael Sullivan.
\newblock Legendrian contact homology in {$P\times\Bbb R$}.
\newblock {\em Trans. Amer. Math. Soc.}, 359(7):3301--3335, 2007.

\bibitem[EGH00]{sft}
Y.~Eliashberg, A.~Givental, and H.~Hofer.
\newblock Introduction to symplectic field theory.
\newblock {\em Geom. Funct. Anal.}, (Special Volume, Part II):560--673, 2000.
\newblock GAFA 2000 (Tel Aviv, 1999).

\bibitem[EK08]{tobkal}
Tobias Ekholm and Tam{\'a}s K{\'a}lm{\'a}n.
\newblock Isotopies of {L}egendrian 1-knots and {L}egendrian 2-tori.
\newblock {\em J. Symplectic Geom.}, 6(4):407--460, 2008.

\bibitem[Ekh07]{trees}
Tobias Ekholm.
\newblock Morse flow trees and {L}egendrian contact homology in 1-jet spaces.
\newblock {\em Geom. Topol.}, 11:1083--1224, 2007.

\bibitem[Ekh08]{ratsft}
Tobias Ekholm.
\newblock Rational symplectic field theory over {$\mathbb{Z}_2$} for exact
  {L}agrangian cobordisms.
\newblock {\em J. Eur. Math. Soc. (JEMS)}, 10(3):641--704, 2008.

\bibitem[{Ekh}19]{tobdisk}
Tobias {Ekholm}.
\newblock {Holomorphic curves for Legendrian surgery}.
\newblock {\em arXiv e-prints}, page arXiv:1906.07228, June 2019, 1906.07228.

\bibitem[EL17]{etgulekili1}
Tolga Etg\"{u} and Yank\i Lekili.
\newblock Koszul duality patterns in {F}loer theory.
\newblock {\em Geom. Topol.}, 21(6):3313--3389, 2017.

\bibitem[EL19]{etgulekili2}
Tolga Etg\"{u} and Yank\i Lekili.
\newblock Fukaya categories of plumbings and multiplicative preprojective
  algebras.
\newblock {\em Quantum Topol.}, 10(4):777--813, 2019.

\bibitem[EN15]{en}
Tobias Ekholm and Lenhard Ng.
\newblock Legendrian contact homology in the boundary of a subcritical
  {W}einstein 4-manifold.
\newblock {\em J. Differential Geom.}, 101(1):67--157, 2015.

\bibitem[GPS20]{gapash1}
Sheel Ganatra, John Pardon, and Vivek Shende.
\newblock Covariantly functorial wrapped {F}loer theory on {L}iouville sectors.
\newblock {\em Publ. Math. Inst. Hautes \'{E}tudes Sci.}, 131:73--200, 2020.

\bibitem[{Kar}16]{orienttrees}
Cecilia {Karlsson}.
\newblock {Orientations of Morse flow trees in Legendrian contact homology}.
\newblock {\em arXiv e-prints}, page arXiv:1601.07346, Jan 2016, 1601.07346.

\bibitem[{Kar}17]{korta}
Cecilia {Karlsson}.
\newblock {To compute orientations of Morse flow trees in Legendrian contact
  homology}.
\newblock {\em arXiv e-prints}, page arXiv:1704.05156, Apr 2017, 1704.05156.

\bibitem[Kar20]{teckcob}
Cecilia Karlsson.
\newblock A note on coherent orientations for exact {L}agrangian cobordisms.
\newblock {\em Quantum Topol.}, 11(1):1--54, 2020.

\bibitem[Sab06]{sabloffdipp}
Joshua~M. Sabloff.
\newblock Duality for {L}egendrian contact homology.
\newblock {\em Geom. Topol.}, 10:2351--2381, 2006.

\bibitem[SW06]{cotangentloopweber}
D.~A. Salamon and J.~Weber.
\newblock Floer homology and the heat flow.
\newblock {\em Geom. Funct. Anal.}, 16(5):1050--1138, 2006.

\bibitem[{Vit}18]{cotangentloopviterbo}
C~{Viterbo}.
\newblock {Functors and Computations in Floer homology with Applications Part
  II}.
\newblock {\em arXiv e-prints}, page arXiv:1805.01316, May 2018, 1805.01316.

\bibitem[Wei91]{weinstein}
Alan Weinstein.
\newblock Contact surgery and symplectic handlebodies.
\newblock {\em Hokkaido Math. J.}, 20(2):241--251, 1991.

\end{thebibliography}

\end{document}